\documentclass{amsart}
\usepackage{amssymb}
\usepackage[matrix,arrow]{xy}
\usepackage{verbatim}
\usepackage{a4wide}
\usepackage[applemac]{inputenc}
\newtheorem{theorem}{Theorem}[section]
\newtheorem{lemma}[theorem]{Lemma}
\newtheorem{remark}[theorem]{Remark}
\newtheorem{definition}[theorem]{Definition}
\newtheorem{conjecture}[theorem]{Conjecture}
\newtheorem{corollary}[theorem]{Corollary}

\newtheorem{proposition}[theorem]{Proposition}

\def\eps{\varepsilon}
\def\GL{\mathrm{GL}}
\def\SO{\mathrm{SO}}
\def\SU{\mathrm{SU}}

\def\C{\mathbf{C}}
\def\Z{\mathbf{Z}}

\def\F{\mathbf{F}}
\def\Q{\mathbf{Q}}
\def\C{\mathbf{C}}
\def\R{\mathbf{R}}

\def\End{\mathrm{End}}
\def\Ker{\mathrm{Ker}}
\def\Hom{\mathrm{Hom}}

\def\into{\hookrightarrow}
\def\onto{\twoheadrightarrow}

\def\AAA{\ensuremath{\hat{A}^{(2 \frac{1}{2})}}}
\def\ii{\mathrm{i}}
\def\la{\lambda}

\date{October 23, 2012.}

\begin{document}
\centerline{}

\title[Freeness conjecture and the Hessian group]{The freeness conjecture for Hecke algebras of complex reflection groups, and the case of the Hessian group  $G_{26}$}
\author[I.~Marin]{Ivan Marin}
\address{LAMFA, Universit\'e de Picardie-Jules Verne, Amiens, France}
\email{ivan.marin@u-picardie.fr}
\subjclass[2010]{}%Primary 20F36; Secondary 20F55}
\medskip

\begin{abstract} 
We review the state-of-the-art concerning the freeness conjecture stated in the 1990's by Brou\'e, Malle and Rouquier for generic Hecke algebras
associated to complex reflection groups, and in particular we expose in detail one of the main differences
with the ordinary case, namely the lack of $0$-Hecke algebras. We end the paper by proving a new case of this conjecture,
the 
exceptional group called $G_{26}$ in Shephard-Todd classification, namely the largest linear group of automorphisms
of the Hessian configuration.

\end{abstract}

\maketitle

\tableofcontents

\section{Introduction}

Between 1994 and 1998, M. Brou\'e, G. Malle and R. Rouquier introduced a natural generalization of the
generic Iwahori-Hecke algebras, associated not only to a Weyl or Coxeter group, but
to an arbitrary (finite) complex reflection group $W$ (see \cite{BMR0,BMR}). Extending earlier work by Brou\'e and Malle (see \cite{BM}),
they found an adequate definition involving the generalized braid group $B$ associated to $W$.
They stated a number of conjectures, some of them involving the braid group $B$,
some others involving the Hecke algebra $H$. All the conjectures concerning the braid group $B$
have apparently been solved now (see \cite{BESSIS,BESSISMICHEL,DMM}).
The ones concerning the Hecke algebras, on the other hand, are not solved yet for
the finite but rather large number of exceptional groups involved in the Shephard-Todd classification
of irreducible reflection groups. Arguably,
the most basic one is the so-called \emph{freeness conjecture}, which states
that $H$ is free of rank $|W|$ as a module over its ring of definition $R$.

The many other existing conjectures about these generalized Hecke algebras
originate in a program about representation theory of finite groups of Lie type, and involve
notably the existence of a canonical trace; this program also suggests a number of
other properties, including that the center of $H$ should also be a free module,
that it behaves well under base changes, and so on. It is also very important to
be able to compute matrix models for the irreducible representations of $H$. However,
the reason why the freeness conjecture is more basic than the other ones is that, once it
is proved, we can rest on our better knowledge of the world of the associative
algebras which have finite type as modules. This better knowledge includes the possibility of
putting structure constants for the multiplication into a computer and apply various algorithms
in order to improve our understanding of what happens in each case (see \cite{MALLEMICHEL}
for an explanation of how the determination of a canonical trace can be made effective
in this way). Also, it is proved in \cite{GGOR} that, provided that the freeness
conjecture is true, the category of representations of $H$ (actually defined over a larger ring)
is equivalent to a category of representations of a `Cherednik algebra', and this
provides other tools in order to possibly deal with the other conjectures.

The primary goal of this paper was to prove this freeness conjecture, that we call here \emph{the} BMR conjecture
in order to emphasize its central role,
for the case of the exceptional group $G_{26}$, which has rank 3 and is the largest of the two
complex reflection groups groups that appear as symmetry groups of the classical Hessian configuration (theorem \ref{theoG26}).
In addition, in the first part of the paper, we provide some more scholarly work, that we felt were
missing in the literature. This includes the comparison of various versions of the BMR conjecture, and
the algebraization of the powerful argument of Etingof and Rains, which proves a weak version
of the conjecture for all groups of rank 2. We also explore in detail why it does
not seem possible to define an analogue of the 0-Hecke algebra for complex reflection
groups, which is a big difference with the usual (Weyl/Coxeter) case.

In this part, we have tried to be as precise and detailed as possible, at the risk of being pedantic
or boring. One reason for this is that we felt that previous work on this conjecture, whose difficulty
has for a long time been underestimated, has sometimes been sloppy on details. For instance,
the proofs given in \cite{BM} for the groups $G_4$, $G_5$, $G_{12}$ and $G_{25}$ are very
sketchy. In addition,  one caveat
that has repeatedly been overlooked for years is the possibility that $H$ might have \emph{torsion}, a phenomenon
that should not happen in view of the conjecture, but which is hard to rule out {\it a priori} --
and this should not be too surprising in view of the example of some torsion elements inside
the `0-Hecke algebras' that we describe below (see figure \ref{figtorsion} and proposition \ref{propG4}). Because of this, one cannot
use embeddings of $H$ into $H \otimes F$, for
$F = Frac(R)$ the field of fractions of $R$.
This mistake appears in \cite{MALLEMICHEL}:
proposition 2.10 is not correct because of this, and this appears to ruin the strategy explained there of
first proving that $H \otimes F$ has finite dimension and deducing from
it the freeness conjecture (in the notations of \cite{MALLEMICHEL}, deducing conjecture 2.2 from conjecture 2.1). It had already appeared in \cite{BM} \S 4B, in the few
details that are given concerning the proof of the BMR conjecture for $G_4$, $G_5$ and $G_{12}$ : the expressions described there have coefficients
which are not specified, but are claimed to belong to $F$,
 which means that
the authors are dealing with $H \otimes F$ instead of $H$. It also appears in \cite{SPETSES},
proposition 2.2 where the uniqueness of the trace is not actually proved over $R$, as it should be in view of  \cite{SPETSES} 2.1 (2),
but over some $R_{k}$, $k \supset \Z$  (see below for notations).

Concerning 
the known cases of the conjecture among the exceptional groups,
the situation thus heavily depends on the standard of rigorousness and checkability you
are willing to accept: depending on this, you can say that either \emph{almost all} or
\emph{almost none} of the exceptional cases have been proved
(with the exception of the weak version proved by Etingof and Rains for groups of rank 2).
On the 
lax
side, one may say that Brou\'e and Malle proved it for the 4 groups above (Berceanu and Funar independently did the case
of $G_4$ in \cite{FUNAR}, appendix A), that I proved it for $G_{32}$ (rank 4) in \cite{conjcub}, and J. Müller
has announced results 10 years ago involving Linton's algorithm
called `vector enumeration', claiming the result for all groups of rank 2 but $G_{17}$, $G_{18}$, $G_{19}$,
as well as all the groups of rank 3 (including $G_{26}$). The missing cases in large rank would then be $G_{29}$, $G_{31}$, $G_{33}$ and $G_{34}$. On the uncompromising side however, we already mentioned possible mistakes in
\cite{BM}, and Müller's program has not been made publicly available and checkable (so far, none of
the usual software in computer algebra implements vector enumeration \emph{over $R$}).
This is problematic in view, not only of the possible mistakes mentioned above, not only because
the need to trust the scientist's word  is something modern science has been trying to avoid for centuries, but also because of the
very nature of the vector enumeration algorithm. This algorithm is indeed a (clever) variation of the Todd-Coxeter algorithm, and
as such provides no control \emph{a priori} on when it stops if it does. Moreover, the moment it ends depends
a lot on a number of heuristic choices that need to be made inside the specific
implementation of the algorithm.
According to J. Müller (private conversation, Aachen, 2010), it is moreover unclear that the program he wrote would run now on modern computers.

In the current situation, running the welcome risk of being checked and judged with the same severity, we thus stick to the
hardline position, 
 of somewhat provocatively claiming that only the cases of $G_4$, $G_{25}$, $G_{32}$ and now $G_{26}$ have been fully
checked so far. Our hope is to 
encourage people to treat the other cases `by hand', which is a way that usually provides
more information about the algebras under consideration, and also to encourage authors and editors to provide and publish full details for these computations. This would enable people to check and
improve these, or use them in possibly unpredicted ways.

\section{General considerations on the generic Hecke algebras}

Let $W$ be a (pseudo-)complex reflection group, always assumed to be finite, and let $R =  
\Z[ a_{s,i}, a_{s,0}^{-1}] $ where $s$ runs
among a (finite) representative system $\mathcal{P}$ of conjugacy classes
of distinguished reflections
in $W$ and $0 \leq i \leq o(s)-1$, where $o(s)$ is the order of $s$ in $W$.
We let $B$ denote the braid group of $W$, as defined in \cite{BMR}, and recall that
a reflection $s$ is called distinguished if its only nontrivial eigenvalue is $\exp(2 \ii \pi/o(s))$, where
$\ii \in \C$ is the chosen square root of $-1$.
% and $d$ is the order of $s$.

\begin{definition} The generic Hecke algebra is the quotient of the group algebra $R B$ by the relations
$\sigma^{o(s)} - a_{s,o(s)-1}\sigma^{o(s)-1} - \dots - a_{s,0} = 0$ for each braided reflection $\sigma$ associated
to $s$.
\end{definition}

Actually,
it is enough to choose one such relation per conjugacy class of distinguished reflection, as all the
corresponding braided reflections are conjugated one to the other.
In \cite{BMR} was stated the following conjecture.

\begin{conjecture} (BMR conjecture) The generic Hecke algebra $H$ is a free $R$-module of rank $|W|$.
\end{conjecture}

\subsection{Root parameters vs. coefficient parameters}

Usually, the Hecke algebra associated to a complex reflection group
is defined over the ring $\tilde{R} = \Z[u_{s,i},u_{s,i}^{-1}]$.
More precisely,
this Hecke algebra $\tilde{H}$ is
defined as the quotient of the group algebra $\tilde{R} B$ of the braid group by
the ideal generated by the relations $\prod_{i=0}^{o(s)-1} (\sigma - u_{s,i}) = 0$
where the $\sigma$ are the braided reflections associated to $s$.
Note that $R$ is the subring $\tilde{R}^{\mathfrak{S}}$ of invariants of $\tilde{R}$ under the natural action
of $\mathfrak{S} = \prod_{s \in \mathcal{P}} \mathfrak{S}_{o(s)}$.

The Broué-Malle-Rouquier conjecture for $H$, namely that $H$ is a free $R$-module of rank $|W|$,
clearly implies that $\tilde{H}$ is a free $\tilde{R}$-module of rank $|W|$. The converse being less obvious, and since
most authors
including \cite{BM} use the Hecke algebra over $\tilde{R}$ instead of $R$, we
prove it here. Note that J. M\"uller in \cite{JMUL} uses the definition over $R$ for his computer calculations.

\begin{lemma} $\tilde{H}$ is a free $\tilde{R}$-module of rank $N$ if and only if $H$ is a free $R$-module of rank $N$.
\end{lemma}
\begin{proof}
Let $I$ denote the two-sided ideal of $R B$ generated by the
$\sigma^{o(s)} - a_{s,o(s)-1}\sigma^{o(s)-1} - \dots - a_{s,0}$.
By definition $\tilde{H}$ is the quotient of $\tilde{R} B = R B \otimes_{R} \tilde{R}$
by the image of $I \otimes_{R} \tilde{R}$ in $R B \otimes_{R} \tilde{R}$
which, by right-exactness of the tensor product, is $H \otimes_{R} \tilde{R}$.
Thus $\tilde{H} = H \otimes_{R} \tilde{R}$. If $H = R^N$ then
clearly $\tilde{H} = \tilde{R}^N \otimes_{R} \tilde{R} = \tilde{R}^N$ is free of rank $N$.

Conversely, we assume that $\tilde{H} = \tilde{R}^N$.
First note that $\tilde{R}$
is a free $R$-module of rank $|\mathfrak{S}|$
(see e.g. \cite{BOURBALG}
%N. Bourbaki, Algèbre, 
chapitre 4, \S 6, n$^o$ 1, th\'eor\`eme 1), hence
$\tilde{H} = R^{N |\mathfrak{S}|}$ as a $R$-module.
Moreover, $\tilde{H} = H \otimes_{R} \tilde{R} \simeq  H \otimes_{R} R^{|\mathfrak{S}|}
 \simeq H^{|\mathfrak{S}|}$ as an $R$-module. In particular $H$ is a direct factor of the
 free $R$-module $\tilde{H}$ and is thus projective, hence flat, as an $R$-module.
 
 Now note that $H = H \otimes_{R} R = H \otimes_{R} \tilde{R}^{\mathfrak{S}} = H \otimes_{R} 
  \Hom_{R \mathfrak{S}}(R,\tilde{R})$
 where $R$ is considered as a trivial $R \mathfrak{S}$-module.
\begin{comment} where $A_0$ is considered as a trivial $A_0 \mathfrak{S}$-module,
 hence $H_0 = \Hom_{A_0 \mathfrak{S}}(A_0, H_0 \otimes_{A_0} A)$,
 where $H_0$ in the RHS is considered as a $A_0 \mathfrak{S}$-module with
 trivial action of $\mathfrak{S}$. 
 \end{comment}
 Since $R$ is a noetherian ring,
 $R \mathfrak{S}$ is noetherian as an $R$-module and thus left-noetherian as a ring,
 hence $R$ admits a finite presentation as an $R \mathfrak{S}$-module.
 Since $R$ is flat, it follows from general arguments (see \cite{BOURBAC1} ch. 1, \S 2, num\'ero 9, prop. 10)
 that 
 $$
 \Hom_{R \mathfrak{S}}(R, H \otimes_{R}\tilde{R}) \simeq H \otimes_{R} \Hom_{R \mathfrak{S}}(R,\tilde{R})
= H \otimes_{R} \tilde{R}^{\mathfrak{S}} = H. $$
But the LHS is $ \Hom_{R \mathfrak{S}}(R, \tilde{H}) =  \Hom_{R \mathfrak{S}}(R, \tilde{R}^N) = 
\Hom_{R \mathfrak{S}}(R, \tilde{R})^N = R^N$ and this proves the claim.
\end{proof}

\subsection{The BMR conjecture and Tits' deformation theorem}

Let $F = \Q(a_{s,i})$ denote the field of fractions of $R$, and $\overline{F}$ an
algebraic closure. 
For $k$ a unital ring, we let $R_{k} = R \otimes_{\Z} k$, $H_{k} = H \otimes_{R} R_{k}$. We let
$\C$ denote the field of complex numbers.
Part (1) of the next proposition is in \cite{BMR} (see the proof of theorem 4.24 there).

\begin{proposition} {\ }

\begin{enumerate}
\item If $H$ is generated as a module over $R$ by $|W|$ elements, then $H$ is a free $R$-module of rank $|W|$.
\item If $H$ is finitely generated as a module over $R$, then $H \otimes_{R} \overline{F} \simeq \overline{F} W$.
\item If $H_{\C}$ is finitely generated as a module over $R_{\C}$, then $H \otimes_{R} \overline{F} \simeq \overline{F} W$.
\end{enumerate}

\end{proposition}
\begin{proof} Let $\mathcal{O} = \C[[h]]$, and $K = \C((h))$ its field of fractions. By the Cherednik
monodromy construction, one can build  a morphism $\varphi : H \otimes \mathcal{O} \to \mathcal{O} W \simeq \mathcal{O}^{|W|}$,
where the morphism $R \to \mathcal{O}$ defining the tensor product $H \otimes \mathcal{O}$ depends on the collection
of parameters involved in the monodromy construction. We take these parameters to be linearly independent
over $\Q$, so that the morphism $R \to \mathcal{O}$ is injective. 
Modulo $h$, the morphism $\varphi$ is the identity of $\C W$, hence the original morphism is surjective by Nakayama's lemma.

Let $N = |W|$.
If $H$ is generated by $N$ elements, then there exists a surjective morphism of $R$-modules
$\pi : R^N \onto H$, which induces $\pi \otimes \mathcal{O} : \mathcal{O}^N \onto H \otimes \mathcal{O}$.
Then $\varphi \circ (\pi \otimes \mathcal{O}) :  \mathcal{O}^N  \onto \mathcal{O} W \simeq \mathcal{O}^N$. Such a surjective morphism between two free modules of the same finite rank is necessarily an isomorphism, hence $\pi \circ \mathcal{O}$
is injective, and in particular $\pi$ is injective. This means $H \simeq R^N$, that is (1). Clearly (3) implies (2), so we
prove (3).

For this purpose we consider $H_{\mathcal{O}} = H_{\C} \otimes \mathcal{O}$, and $\pi_{\mathcal{O}} : \mathcal{O}^{N'} \onto
H_{\mathcal{O}}$. Let $H_{\mathcal{O}}^0$ the torsion submodule of $H_{\mathcal{O}}$. Since
$H_{\mathcal{O}}$ is finitely generated over the principal ring $\mathcal{O}$,
we have that $\hat{H_{\mathcal{O}}} = H_{\mathcal{O}}/H_{\mathcal{O}}^0$
is a finitely generated \emph{free} $\mathcal{O}$-module.
$$
\xymatrix{
H_{\mathcal{O}} \ar@{->>}[r]^{\varphi} & \mathcal{O} W \ar[d]^{h=0} \\
H_{\mathcal{O}}^0 \ar@{^{(}->}[u] \ar[ur]^{0} \ar[r] & \C W 
}
$$
By the above commutative diagram, the specialization morphism $H_{\mathcal{O}} \onto \C W$ factors through $\hat{H_{\mathcal{O}}}$,
hence we can apply Tits' deformation theorem to the free $\mathcal{O}$-module $\hat{H}_{\mathcal{O}}$ and get $\hat{H_{\mathcal{O}}}  \otimes \overline{K} \simeq \overline{K} W$.
Since $\hat{H_{\mathcal{O}}}  \otimes \overline{K} \simeq H_{\mathcal{O}} \otimes \overline{K}$ this
means $H_{\mathcal{O}} \otimes \overline{K} \simeq \overline{K} W$, 
hence $H \otimes \overline{F}$ is a semisimple algebra with the same numerical invariants (because
$H_{\mathcal{O}} \otimes \overline{K} = H \otimes \overline{K}$ is an extension of scalars, since $\overline{F} \into
\overline{K}$), and this implies
%whence
$H \otimes \overline{F} \simeq \overline{F} W$.

%and
%because $ H_{\mathcal{O}} \otimes \overline{K} \simeq H \otimes \overline{F} \otimes \overline{K}$ \dots

%Under $\mathcal{O}^{N'} \onto H_{\mathcal{O}}/
%H_{\mathcal{O}}^0 \simeq \mathcal{O}^N$ we get a preimage $\mathcal{O}^N \subset \mathcal{O}^{N'}$.
%We can now apply Tits

%By exactness of $M \otimes \C \to R^{N'} \otimes \C \to H \otimes \C = \C W \to 0$, we can pick
%$m_1,\dots,m_{N'-N}$ linearly independent  in $M \otimes \C$ whose image in $R$
\end{proof}

Recall that every finitely generated flat $R$-module is projective ; moreover, as a consequence of Swan's big rank
theorem (see \cite{SWAN}, and \cite{LAM}, ch. 5) every finitely
generated projective $R$-module of rank at least $2$ is actually free and, because $\Z$
is a regular ring, $K_0 (R) \simeq K_0(\Z) = \Z$ which implies that also the rank 1 projective
modules are free (see e.g. \cite{LAM}, ch. 5 lemma 4.4 and ch.1 cor. 6.7). We summarize this as follows.

\begin{proposition} If $H$ is finitely generated as an $R$-module, then
$H$ is free of rank $|W|$ if and only if it is flat.
\end{proposition}

 If it is known that $H$ is finitely generated, the BMR conjecture thus becomes a local condition. For
 a prime ideal $\mathfrak{p}$ of $R$, and $M$ an $R$-module, let $R_{\mathfrak{p}}$ denote
 the localization of $R$ at $\mathfrak{p}$ and $M_{\mathfrak{p}} = M \otimes_{R} R_{\mathfrak{p}}$.

In the specific neighborhoods of $\mathrm{Spec}\, R$ corresponding to the specializations $H  \to k W$,
the following can be proved.

\begin{proposition} Let $k$ be a field. 
Let $\mathfrak{m} = \Ker(R \to k)$ be the maximal ideal defined by $a_{s,i} \mapsto  0$ for $i >0$ and $a_{s,0} \mapsto 1$.
 If $H_{\mathfrak{m}}$ is finitely generated as a $R_{\mathfrak{m}}$-module (for instance if $H$ is finitely
 generated as a $R$-module), then it
% $H_{\mathfrak{m}}$ 
 is free of rank $|W|$.
 \end{proposition}
%\begin{proof}

%\end{proof} 
%Let $\widehat{H}_{\mathfrak{m}}$ and $\widehat{R}_{\mathfrak{m}}$
%denote the completions of $H_{\mathfrak{m}}$ and $R_{\mathfrak{m}}$ with respect to
%the $\mathfrak{m}$-adic topology, for the chosen ideal $\mathfrak{m}$.
 The proposition above is a consequence
(see e.g. \cite{BOURBAC1} ch. III \S 3 no. 5, cor. 2 prop. 9) of the next one,
basically deduced from \cite{BMR} by Etingof and Rouquier (unpublished).
This property is essentially what P. Etingof calls `formal flatness'.

%\begin{proposition}With $\mathfrak{m}$ as above, 
%if $H_{\mathfrak{m}}$ is a finitely generated
%as a $R_{\mathfrak{m}}$-module, then  
%$\widehat{H_{\mathfrak{m}}}$ 
%is a free $\widehat{R_{\mathfrak{m}}}$-module  of
%rank $|W|$.
%\end{proposition}
\begin{proposition}
Let $\mathfrak{p} = \Ker( R \to \Z)$ be the prime ideal defined by $a_{s,i} \mapsto  0$ for $i >0$ and $a_{s,0} \mapsto 1$, and let $\widehat{R}$, $\widehat{H}$, etc., 
denote the completions w.r.t. the $\mathfrak{p}$-adic topology.
Then $\widehat{H}$ is a free $\widehat{R}$-module of rank $|W|$.
%With $\mathfrak{m}$ as above, 
%if $H_{\mathfrak{m}}$ is a finitely generated
%as a $R_{\mathfrak{m}}$-module, then  
%$\widehat{H_{\mathfrak{m}}}$ 
%is a free $\widehat{R_{\mathfrak{m}}}$-module  of
%rank $|W|$.
\end{proposition}

\begin{proof}
Let $\tilde{W} = \{ \tilde{w} \ | \ w \in W \}$ be the image of a set-theoretic section of $B \onto W$, and $\check{W}$ be its image
in $H$.
%and $\mathfrak{p} = \Ker( R \to \Z)$. It is a prime ideal contained in $\mathfrak{m}$.
The map $w \mapsto \tilde{w}$ induces using the natural projection $R B \onto H$ a  continuous morphism of 
\begin{comment}
$R_{\mathfrak{m}}$-modules
$R_{\mathfrak{m}} W \to \widehat{H_{\mathfrak{m}}}$,
hence $\Phi_{\mathfrak{m}} :\widehat{R_{\mathfrak{m}}} W \simeq  \widehat{R_{\mathfrak{m}} W} \to \widehat{H_{\mathfrak{m}}} $.
We prove that $\Phi_{\mathfrak{m}}$ is surjective. We first need a lemma.
\end{comment}
\begin{comment}
$R_{\mathfrak{p}}$-modules
$R_{\mathfrak{p}} W \to \widehat{H_{\mathfrak{p}}}$,
hence $\Phi_{\mathfrak{p}} :\widehat{R_{\mathfrak{m}}} W \simeq  \widehat{R_{\mathfrak{p}} W} \to \widehat{H_{\mathfrak{p}}} $.
We prove that $\Phi_{\mathfrak{p}}$ is surjective. We first need a lemma.
\end{comment}
$R$-modules
$R W \to \widehat{H}$,
hence $\Phi :\widehat{R} W \simeq  \widehat{R W} \to \widehat{H} $.
We prove that $\Phi$ is surjective. We first need a lemma.

\begin{lemma} 
For all $r \geq 1$, $H = R \check{W} + \mathfrak{p}^r H$.
\end{lemma}
\begin{proof}
Let $P = \Ker(B \to W)$ denote the pure braid group.
We have $R B = R \tilde{W} + \tilde{W} \sum_{g \in P} (g-1)$. 
The image in $H$ of $(g-1)$ for $g \in P$ falls into 
%In $H$, the $(g-1)$ for $g \in P$
%are mapped to 
$\mathfrak{p} H$, 
hence $H = R \check{W} + \mathfrak{p} H$ which easily implies
the conclusion.
\end{proof}
%As a consequence, we get that $R_{\mathfrak{m}} \tilde{W}$ is dense inside $\widehat{H_{\mathfrak{m}}}$,
As a consequence, we get that $R \check{W}$ is dense inside $\widehat{H}$,
thus proving that $\Phi$ is surjective. 
%More generally, letting $\widehat{R}$ resp. $\widehat{R W}$  denote the completion
%of $R$ resp. $R W$ with respect to the $\mathfrak{p}$-adic topology, we get that $\widehat{R} W \simeq \widehat{R W} \to \widehat{H}$
%is surjective. 
Now, the KZ-like construction of \cite{BMR} provides a morphism $\psi : H \to \C[[h]] W$
associated to some morphism $R \to \C[[h]]$, which is continuous for the $(\mathfrak{p}, h)$-topologies, hence extends to
a morphism $\widehat{H} \to \C[[h]]W$. We have the following diagram.
%The morphism $R \to \C[[h]]$ can be chosen to
$$\xymatrix{
\widehat{R} W \ar@{->>}[r] \ar[dd]^{\mathfrak{p}=0}& \widehat{H} \ar[dr] \\ 
& & \C[[h]] W  \ar[dl]_{h=0} \\
\Z W \ar@{^{(}->}[r] & \C W \\
}
$$
Since $\widehat{R} W$ is a free module of finite rank, the injectivity of
the natural map $\Z W \to \C W$ implies that $\widehat{R} W \to \widehat{H}$ is
injective, whence  $\widehat{R} W \simeq \widehat{H}$ and the conclusion.

\end{proof}

An elementary remark is that, in a number of exceptional cases, the BMR conjecture
can be reduced to a problem in 1 variable. Indeed, the following apply to 
all the exceptional non-Coxeter groups of rank at least $3$, except the
Shephard group $G_{25}$, $G_{26}$ and $G_{32}$, that is to
the
groups $G_{24}$, $G_{27}$, $G_{29}$, $G_{31}$, $G_{33}$, $G_{34}$.  We let $\mathcal{A}$
denote the collection of all the hyperplanes which are sets of fixed points for reflections in $W$.
It is naturally acted upon by $W$.

\begin{proposition} Assume that $| \mathcal{A}/W| = 1$ and that all reflections have order $2$.
Let $k$ be a commutative
unital ring. Then $H_k$ is a free $R_k$-module of rank $|W|$ (respectively, a finitely generated $R_k$-module)
if and only if the quotient of $k[x^{\pm 1}] B$ by the ideal generated by $(\sigma -1)(\sigma-x)$
for $\sigma$ a braided reflection, has the same property.
\end{proposition}
\begin{proof} Let $H^0$ denote the quotient of $k[x^{\pm 1}] B$ by the ideal generated by $(\sigma -1)(\sigma-x)$
for $\sigma$ one given braided reflection. Since all braided reflections are conjugate in $B$, this
is the algebra involved in the statement. Moreover notice that $B$ is normally generated, as a group, by $\sigma$. 
Let $\tilde{R}_k = k[u_0^{\pm 1}, u_1^{\pm} ]$. We already know that the assumptions of the proposition
on $H_k$ are equivalent to the same assumptions for $\tilde{H}_k = H_k \otimes_{R_k} \tilde{R}_k$.
For every $\beta \in \tilde{R}_k$, there exists an algebra morphism
$\varphi_{\beta} : \tilde{R}_k B \to \tilde{R}_k B$ which maps $\sigma$ to $\beta \sigma$ ; it can be defined, using that $B^{ab} \simeq \Z$
is generated by the image $\overline{\sigma}$ of $\sigma$, as the composite of the natural algebra morphisms
$$
\xymatrix{
\tilde{R}_k B \ar[r]_{\!\!\!\!\!\!\!\!\! \!\!\!\!\!\!\!\!\! \Delta}& (\tilde{R}_k B) \otimes (\tilde{R}_k B) \ar[r]_{\mathrm{Id} \otimes Ab} & (\tilde{R}_kB)  \otimes (\tilde{R}_k B^{ab}) 
\ar[r]_{\overline{\sigma} \mapsto \beta} & (\tilde{R}_k B) \otimes \tilde{R}_k \ar[r]_{\ \ \ \ \simeq} & \tilde{R}_k B
}
$$
where $\Delta$ denotes the coproduct. These morphisms are equivariant under the conjugation action of $B$
on itself, thus $\varphi_{\beta_1} \circ \varphi_{\beta_2} (\sigma) = \varphi_{\beta_1 \beta_2} (\sigma)$
implies that $\varphi_{\beta} \circ \varphi_{\beta^{-1}} = \mathrm{Id}$, hence $\varphi_{\beta}$ is an isomorphism
when $\beta \in \tilde{R}_k^{\times}$. Recall now that $\tilde{H}_k$ is the quotient of $\tilde{R}_k B$ by the
relation $(\sigma - u_0)(\sigma-u_1)$. Taking $\beta = u_0$, we get that the image of $(\sigma-u_0)(\sigma - u_1)$
under $\varphi_{u_0}$ is $u_0^2 (\sigma-1)(\sigma - u_1 u_0^{-1})$, whence a $\tilde{R}_k$-isomorphism
$\tilde{H}_k \simeq H^0 \otimes_{k[x^{\pm 1}]} \tilde{R}_k$, where $k[x^{\pm 1}] \subset \tilde{R}_k$
is defined by $x \mapsto u_1 u_0^{-1}$. We have $R_k = \bigoplus_{a,b} k u_0^a u_1^b = \bigoplus h (u_1 u_0^{-1})^b u_0^a$
is free hence faithfully flat as a $k[x^{\pm 1}]$-module, and this proves the claim.
\end{proof}

Let $z \in Z(B)$, $\overline{B} = B / \langle z \rangle$, and $s : \overline{B} \to B$
a set-theoretic section of the natural projection $b \mapsto \overline{b}$. We let $R_k^+ = R_k[x,x^{-1}]$.
We denote $\sigma_1,\dots,\sigma_r$ a distinguished system of braided reflections
(all corresponding to distinguished reflections, and at least 1 for each conjugacy
class). We let $P_i \in R_k[X]$ denote polynomials defining $H_k$,
that is $H_k$ is the quotient of $R_k B$ by the relations $P_i(\sigma_i) = 0$.

\begin{proposition} {\ } \label{propER1}
\begin{enumerate}
\item $R_k B$ admits an $R_k^+$-module structure defined by $x.b = zb$ for $b \in B$. It
is a free $R_k^+$-module, and we have an isomorphism $R_k B \simeq R_k^+ \overline{B}$.
\item Under this isomorphism, the defining ideal of $H_k$
is mapped to the ideal of $R_k^+ \overline{B}$ generated by the $Q_i(\overline{\sigma}_i)$
for $Q_i (X) = P_i(X x^{a_i}) \in R^+[X]$, the $a_i \in \Z$ being defined
by $\sigma_i = s(\overline{\sigma}_i) z^{a_i}$.
\end{enumerate}
\end{proposition}
\begin{proof}
Since $Z(B)$ is torsion-free, $\langle z \rangle \simeq \Z$, and there is a uniquely
defined $1$-cocycle $\alpha : B \onto \Z$ such that $\forall b \in B \ b = z^{\alpha(b)} s(\overline{b})$. Therefore,
as $R_k$-modules, $R_k B = \bigoplus_{b \in B} R_k b = \bigoplus_{b \in B} R_k z^{\alpha(b)} s(\overline{b})
= \bigoplus_{d \in \overline{B}}  \bigoplus_{\overline{b} = d} R_k z^{\alpha(b)} d$. Under this
identification, for every $d \in \overline{B}$, $\bigoplus_{\overline{b} = d} R_k z^{\alpha(b)} d$
is a free $R_k^+$-submodule of rank $1$, whence (1). Let $I \subset R_k B$ the defining
ideal for $H_k$. It is the $R_k$-submodule spanned by the $b P_i(\sigma_i) c$
for $i \in \{1, \dots, r \}$ and $b,c \in B$, hence also the $R^+_k$-submodule
spanned by the $s(d) P_i(\sigma_i) s(e)$ for $d,e \in \overline{B}$. We have $P_i(\sigma_i)
= Q(s(\overline{\sigma}_i))$, hence $I$ is identified inside $R_k^+ \overline{B}$ with the
$R_k^+$-submodule generated by the $d Q_i(\overline{\sigma}_i)e$  for $d,e \in \overline{B}$,
that is to the ideal of $R_k^+ \overline{B}$ generated by the $Q_i(\overline{\sigma}_i)$.
\end{proof}

Note that, if the $P_i$ have been chosen to be monic, one can replace the $Q_i$ by
the monic polynomials $x^{-a_i d^{\circ} P_i} P_i(X x^{a_i}) \in R_k^+[X]$, where $d^{\circ} P_i$ denotes
the degree of $P_i$. Note also that,
being the quotient of $R_k B$ by an $R_k^+$-submodule,
$H_k$ inherits a structure of $R_k^+$-module.

The following proposition is based on  one of the arguments of Etingof-Rains \cite{ETINGOFRAINS0}.
\begin{proposition} \label{propER2}
If $H_k$ is finitely generated as a $R_k^+$-module, then it is finitely
generated as a $R_k$-module.
\end{proposition}
\begin{proof} By assumption, $H_k$ is generated as a $R_k^+$-module by a finite
set of $N_1$ elements.
Let us choose some $Q \in R_k[x] \subset R_k^+$, and $M = Q H_k \subset H_k$. By assumption it
is generated by $N_1$ elements as a $R_k^+$-module. We have $M = 0$
as soon as $M_{\mathfrak{a}} = 0$ for all the maximal ideals $\mathfrak{a}$ of $R_k^+$,
where $(R_k^+)_{\mathfrak{a}}$ denotes the localization of $R_k^+$ at $\mathfrak{a}$, 
and $M_{\mathfrak{a}} = M \otimes_{R_k^+}  (R_k^+)_{\mathfrak{a}}$ (see e.g. \cite{EISENBUD} lemma 2.8, p. 67-68).
Since $(R_k^+)_{\mathfrak{a}}$ is a local ring, by Nakayama's lemma we have $M_{\mathfrak{a}} = 0$
iff $M_{\mathfrak{a}} = \mathfrak{a} M_{\mathfrak{a}}$. Since $(R_k^+)_{\mathfrak{a}}$ is a flat $R_k^+$-module,
this is equivalent to $M = \mathfrak{a} M$, that is $M \otimes_{R_k^+} ((R_k^+)/\mathfrak{a})=0$. Let us denote
$K = (R_k^+)/\mathfrak{a}$ and $\overline{K}$ the algebraic closure of $K$, and $\tilde{\la} : R_k^+ \to \overline{K}$
the natural morphism. We have $M \otimes_{R_k^+} \overline{K} = Q (H_k \otimes_{R_k^+} \overline{K})
= (Q H_k) \otimes_{R_k^+} \overline{K} = H_k \otimes_{R_k^+} \tilde{\la}(Q) \overline{K}$.
Let $\tilde{R}_k = \tilde{R} \otimes_{\Z} k$ and 
$\tilde{R}^+_k = \tilde{R}_k[x,x^{-1}]$. Since it is clearly an integral
extension of $R_k^+$ and $\overline{K}$ is algebraically
closed, $\tilde{\la}$ can be extended to $\tilde{\la} : \tilde{R}^+_k \to \overline{K}$.

We know that $H_k \otimes_{R_k^+} \overline{K}$ is a $\overline{K}$-algebra of dimension at most $N_1$.
If $H_k \otimes_{R_k^+} \tilde{\la}(Q) \overline{K} \neq 0$, there exists a simple $H_k \otimes_{R_k^+}  \overline{K}$-module
$V$ of dimension $N_2(V) \leq N_1$ in which $\tilde{\la}(Q)$ acts by a nonzero scalar. It defines an irreducible
representation $\rho : B \to \GL_{N(V_2)}(\overline{K})$ of $B$ over $\overline{K}$. By definition, $z$ acts on $V$ through $\tilde{\la}(x)$,
which has determinant $\tilde{\la}(x)^{N_2(V)}$. On the other hand, $z$ is the product of $N_3$ distinguished
braided reflections $\sigma_1 \dots \sigma_{N_3}$ with $N_3$ independent of the previous choices.
Since each $\rho(\sigma_i)$ is annihilated by a split polynomial with roots inside $\{ \tilde{\la}(u_{c,i}) \ | \ c \in \mathcal{A}/W, 0 \leq i \leq e_c-1 \}$, $\det \rho(\sigma_i)$ is a monomial of degree $N_2$ in these variables. It follows
that $\tilde{\la}(x)^{N_2(V)}$ is a monomial of degree $N_2(V) N_3$ in these variables. Let $\mathcal{M}$
be the set of all such monomials of degree at most $N_1 N_3$. Since $N_2(V) \leq N_1$, we get
that $\tilde{\la}(x)$ is annihilated by the polynomial
$$
\tilde{Q} = \prod_{1 \leq r \leq N_1} \prod_{m \in \mathcal{M}}(X^r - m) \in (\tilde{R}_k[X])^{\mathfrak{S}} = R_k[X]
$$
so we set $Q = \tilde{Q}(x) \in R_k^+$. By construction we have that $\tilde{\la}(Q) = \tilde{Q}(\tilde{\la}(x))$
acts by $0$ on $H_k \otimes_{R_k^+} \overline{K}$, hence $M = 0$ and $Q H_k = H_k$.
It follows that $H_k = H_k \otimes_{R_k^+} (R_k^+/(Q))$ is finitely generated as an $R_k^+/(Q)$-module.
Since $R_k^+/(Q)$ is a finitely generated $R_k$-module, the conclusion follows.
\end{proof}

\subsection{Groups of rank 2}

Assume that $W$ is an irreducible exceptional group, and that it has rank $2$. This part is a rewriting of the
part of \cite{ETINGOFRAINS0,ETINGOFRAINS1,ETINGOFRAINS2}  which is relevant here.
Let $\overline{B} = B/Z(B)$, and $\overline{W} = W/Z(W)$. A consequence of the classification
of the finite subgroups of $\SO_3(\R) \simeq \SU_2/Z(\SU_2)$ is that
$\overline{W}$ is the group of rotations
of a finite Coxeter group $C$ of rank $3$, with Coxeter system $y_1,y_2,y_3$ and Coxeter matrix $(m_{ij})$.
Let $\tilde{\Z} = \Z[ \exp\left( \frac{2 \ii \pi}{m_{ij}} \right) ]$.

Etingof and Rains associate to every Coxeter group $C$ with Coxeter
system $y_1,\dots,y_n$ the following $\tilde{\Z}$-algebra (for simplicity, we assume $m_{ij} < \infty$,
although their construction is more general). Let $a_{ij} = y_i y_j \in \overline{W}$. 
and define $A(C)$ to be the (associative) algebra with generators
$Y_1,\dots,Y_n$, $t_{ij}[k]$ for $i,j \in \{1,2,3\}$, $i \neq j$, $k \in \Z/m_{ij} \Z$, and relations
$t_{ij}[k]^{-1} = t_{ji}[-k]$, $Y_i^2=1$, $\prod_{k=1}^{m_{ij}} (Y_i Y_j - t_{ij}[k]) = 0$,
$y_r t_{ij}[k] = t_{ji}[k] y_r$, $t_{ij}[k] t_{i'j'}[k'] =t_{i'j'}[k'] t_{ij}[k]$. The subalgebra $A_+(C)$
generated by the $A_{ij} = Y_i Y_j$  becomes a $R^C$-algebra, with
$R^C = \tilde{\Z}[t_{ij}[k]^{\pm 1}]$. As a $R^C$-algebra, it admits a presentation by
generators $A_{ij}$ and relations $\prod_{k=1}^{m_{ij}} (A_{ij} - t_{ij}[k]) = 0$, $A_{ij} A_{ji} = 1$, $A_{ij} A_{jk} A_{ki} = 1$
whenever $\# \{i,j, k \} = 3$.

\begin{proposition}(Etingof-Rains) If $C$ is finite, then $A_+(C)$ is a finitely generated $R^C$-module.
\end{proposition}
\begin{proof} (sketch) Every word in the $A_{ij}$'s corresponds to a word (of even length) in the $y_i$'s ; if the length
of this word is greater than the length of the corresponding element of $C$, then there
is a sequence of braid relations that transforms this word into another one containing $y_i^2$ for some $i$.
Moreover, it is easily checked that every braid relation can be translated inside $A_+(C)$ into the
transformations $A_{ij}^{\frac{m_{ij}}{2}} \leadsto
A_{ji}^{\frac{m_{ij}}{2}} + \dots$ or $A_{ij}^{(m_{ij}-1)/2}A_{i\ell} \leadsto A_{ji}^{(m_{ij}-1)/2}A_{j\ell} + \dots$
where the dots represent terms of smaller length. Finally, when the word in the $y_i$'s contains a $y_j^2$,
this means that our original word contains a $A_{ij} A_{ik}$, which is either $1$ or $A_{ik}$, hence the
length gets reduced. This proves that $A_+(C)$ is generated as a $R^C$-module by words
of bounded length in the $A_{ij}$'s, hence that it is finitely generated as a $R^C$-module.
\end{proof}

In order to apply this to our $W$, Etingof and Rains exhibit case-by-case in \cite{ETINGOFRAINS0} explicit lifts $\tilde{a}_{ij} \in \overline{B}$, using
which they prove the following, where we use the notations of proposition \ref{propER1}.

\begin{proposition} There exists a ring morphism $R^C \onto R^+_{\tilde{\Z}}$
inducing $A_+(C) \otimes_{R^C} R^+_{\tilde{\Z}} \onto R^+_{\tilde{\Z}} \overline{B}/Q_i(\overline{\sigma}_i)$ through
$A_{ij} \mapsto \tilde{a}_{ij}$.
\end{proposition}

An immediate consequence of this proposition together with propositions \ref{propER1} and \ref{propER2} is
the following, essentially due to Etingof and Rains.
\begin{theorem} %(Etingof-Rains)
If $W$ has rank 2, then $H$ is finitely generated over $R$.
\end{theorem}
\begin{proof}
By the above proposition and propositions \ref{propER1} and \ref{propER2} we get that
$H_{\tilde{\Z}}$ is finitely generated over $R_{\tilde{\Z}}$. Since $\tilde{\Z}$ is a free $\Z$-module
of finite rank (being finitely generated and torsion-free)
$R_{\tilde{\Z}}$ is also free of finite rank over $R_{\Z}$. This implies that
$H \subset H_{\tilde{\Z}}$ and that $H_{\tilde{\Z}}$ is finitely generated over $R$.
Since $R$ is noetherian this implies the conclusion.
\end{proof}

\begin{remark}
There are exceptional groups of higher rank which are related to Coxeter groups,
notably $G_{33}/Z(G_{33})$ and $G_{32}/Z(G_{32})$ are isomorphic to the group of rotations of the Coxeter
group of type $E_6$ (which is a simple group of order $25920$). One has $|Z(G_{33})| =2$,
$|Z(G_{32})| = 6$. However, the same method does not readily apply, because of the lack of convenient
lifts.
\end{remark}

\section{Remarks on the 0-Hecke algebras}

In the Coxeter case, there is a notion of a 0-Hecke algebra which, although not being
the quotient of the group algebra of $B$ anymore, nevertheless displays many
pleasant properties. In particular, it is still a free module of rank the order
of the Coxeter group, and it admits an interpretation as an algebra of differential
operators. In this section we expose two different kinds of obstructions for
such a nice behavior to generalize.

\subsection{Demazure operators}

In \cite{LEHRTAY}, problem 5 in appendix C, G. Lehrer and D.E. Taylor ask
whether the `Demazure operators', which provide a description of the $0$-Hecke
algebra in the Coxeter setting, may provide a satisfactory description of the $0$-Hecke
algebra in the complex setting. More precisely, they ask whether these operators
satisfy the homogeneous relations originating from the usual braid diagrams of the braid group.

In this section we give a negative answer to this problem, by computing precisely these
operators in the smallest exceptional case, namely of the exceptional reflection group of type $G_4$. Recall
e.g. from \cite{BMR} that this group admits a Coxeter-like diagram of the form

\def\nnode#1{{\kern -0.6pt\mathop\bigcirc\limits_{#1}\kern -1pt}}
\def\ncnode#1#2{{\kern -0.4pt\mathop\bigcirc\limits_{#2}\kern-8.6pt{\scriptstyle#1}\kern 2.3pt}}
\def\sbar#1pt{{\vrule width#1pt height3pt depth-2pt}}
\def\dbar#1pt{{\rlap{\vrule width#1pt height2pt depth-1pt} 
                 \vrule width#1pt height4pt depth-3pt}}

$$
G_{4}\ \ \  \ncnode3{s_1}\sbar16pt\ncnode3{s_2}\ \ \ \ \ \ \ \ \ \ \ \ 
$$
meaning that its braid group $B$ is generated by two braided reflections $s_1$ and $s_2$
with relations $s_1 s_2 s_1 = s_2 s_1 s_2$ (hence $B$ is isomorphic to the usual braid group
on $3$ strands, or Artin group of type $A_2$), and that the reflection group %$W < \GL_2(\C)$ 
itself is the quotient
of $B$ by the relations $s_1^3 = s_2^3 = 1$.

The defining embedding $W < \GL_2(\C)$ can for instance be described as follows, with $j = \exp(2 \ii \pi/3)$, $\ii = \sqrt{-1}$. % (see e.g. \cite{AUTGAL}) 
$$
s_1 = \begin{pmatrix} 1 & 0 \\ 0 & j \end{pmatrix} \ \ 
s_2 = \frac{1}{3} \begin{pmatrix} j-j^2 & 2j+j^2 \\ 4j+2j^2 & -j-2j^2 \end{pmatrix}
$$
%(ce sont les modèles de AutGal, en remplaçant $\sqrt{-3}$ par sa valeur).
Let $\delta^*_i \in \End(S(V^*))$ denote the Demazure operator associated to $s_i$, and $\delta_i \in \End(S(V))$
its dual operator. It is true that $(\delta_1^*)^3 = (\delta_2^*)^3 = 0$ (see e.g. \cite{LEHRTAY}, chapter 9,
exercises), and the general question 
specializes to
whether $\delta_1^* \delta_2^* \delta_1^* = \delta_2^* \delta_1^* \delta_2^*$ holds, or equivalently 
whether $\delta_1 \delta_2  \delta_1  = \delta_2  \delta_1  \delta_2 $ holds, possibly up to a renormalization
of the operators by non-zero scalars.
%On considère l'action sur $S(V)$, et non pas sur $S(V^*)$. Ca ne change rien.
We now explain the following computation.

\begin{proposition} $\delta_1 \delta_2 \delta_1 \not\in \Q(j)^{\times} \delta_2 \delta_1 \delta_2$, hence the Demazure operators associated
to the braid diagram of $G_4$ do not satisfy the braid relations up to a scalar.
\end{proposition}
\begin{proof}
We let $V = \C^2$ with canonical basis 
$x = e_1$, $y = e_2$, hence  $S(V) = \C[x,y]$, and $s_2$ maps $x$ on $((j-j^2) x + (4j + 2j^2) y)/3$, etc.
The reflecting hyperplane of $s_1$ is spanned by $x$, and its root is a multiple of $y$ ;
the reflecting hyperplane of $s_2$ is spanned by $x-2y$, and its root is a multiple of $x+y$.
Thus the corresponding Demazure operators are defined, up to a scalar of degree $0$,
by %Les opérateurs de Demazure associés sont donc définis, à un scalaire de degré $0$ près,
%par 
 $s_1.p - p = y \delta_{1} p$, $s_2.p - p = (x+y) \delta_{2} p$. %On doit avoir $\delta_{s_1}^3
%=\delta_{s_2}^3 = 0$.
%$\delta_{s_1} : x^a y^0 \mapsto 0$,
The expression of $\delta_{1}$ is simple, as it maps a monomial of the form
$ x^a y^b$ to $(j^b -1) x^a y^{b-1}$, as shown by a simple induction.
%direct though lengthy computation.
% (calcul direct)
%celui de $\delta_{s_2}$ ne l'est pas.
%$\delta_{s_2} : x^a y^b \mapsto (j^b -1) x^a y^{b-1}$ (calcul direct)
The computation of $\delta_2$ is more intricate. One gets easily
$$\left\lbrace
\begin{array}{lcr}
3 \delta_2.y &=& 2j+j^2 \\
3 \delta_2. x &=& 4j+2j^2 \\
\end{array} \right.
\left\lbrace\begin{array}{lcrcr}
3 \delta_2.y^2 &=&  -j x &-& (3 +  j^2) y \\
3 \delta_2. xy &=& j^2 x  &-& 2 y \\
3 \delta_2. x^2 &=&  -4 x &-& 4 y
\end{array} \right.
$$
and

$$
\begin{array}{lcl}
9 \delta_2.y^3 &=&  (j-j^2)x^2 - (7j+2j^2) xy + (10j+8j^2) y^2\\
9 \delta_2 . y^4 &=& j^2 x^3 + (4j - j^2) x^2 y  - (10j + 5j^2) x y^2 - (10 + j^2) y^3
\end{array}
$$
Starting from  $\delta_1.y^4 = (j^4-1) y^3= (j-1) y^3$ one thus
gets 
$$3 \delta_1 \delta_2 \delta_1 . y^4 = (5j^2 - 2j) x + (10j + 8 j^2) y
$$
and 
$$ 9 \delta_2 \delta_1 \delta_2 . y^4 = (4j- 13j^2)x + (2j^2 - 2j) y.
$$
This implies that $ \delta_1 \delta_2 \delta_1 . y^4$ and
$ \delta_2 \delta_1 \delta_2 . y^4$ are linearly independent, which
proves the claim.
\end{proof}

Of course, this obstruction might {\it a priori} vanish by taking another kind of diagrams.
However, we notice that all the 6 pairs of the form $\{ s,t \}$ with $s,t$ among the
4 distinguished pseudo-reflections of the
reflection group $G_4$ are conjugate to each other, whence from the
above we get $\delta_s \delta_t \delta_s \not\in \Q(j)^{\times} \delta_t \delta_s \delta_t$ for each
of them. 

After this example was computed, R. Rouquier told the author that M. Brou\'e had already tried,
some twenty years ago, to use Demazure operators for complex reflection groups, and that he had already noticed
a similar defect.

%If there is another diagram giving 
%satisfy in pairs the braid relation $xyx = yxy$,
%and that all paris

\subsection{0-Hecke algebras defined by diagrams}
%The case of $G_4$ : finite generation and torsion}

\subsubsection{The case of $G_4$ : finite generation and torsion}

In view of the diagram described above, a natural candidate
for the $0$-Hecke algebra associated to the reflection group
$G_4$ would be the following algebra.
%the algebra defined by generators

\begin{proposition} Let $k$ be a ring. The unital $k$-algebra defined by generators
$s_1,s_2$ and relations $s_1 s_2 s_1 = s_2 s_1 s_2$, $s_1^3 = s_2^3 = 0$ is not
finitely generated as a $k$-module.
\end{proposition}
\begin{proof}
Let $\mathcal{W},\mathcal{W}',\mathcal{Y},
\mathcal{Y}'$ be free $k$-modules with bases $w_r, w'_r, y_r, y'_r$, $r \geq 1$,
and let $\mathcal{E} = 
 \mathcal{W}\oplus\mathcal{W}'\oplus\mathcal{Y}\oplus
\mathcal{Y}'$. We define $k$-endomorphisms $S_1$ and $S_2$ of $\mathcal{E}$ by
$$
\left\lbrace \begin{array}{lcl} S_1 . w_r &=& 0 \\ S_2 . w_r &=& y_{r+1} \end{array} \right.
\left\lbrace \begin{array}{lcl} S_1 . y_r &=& 0 \\ S_2 . y_r &=& w'_r \end{array} \right.
\left\lbrace \begin{array}{lcl} S_1 . w'_r &=& y'_{r+1} \\ S_2 . w'_r &=& 0 \end{array} \right.
\left\lbrace \begin{array}{lcl} S_1 . y'_r &=& w_r \\ S_2 . y'_r &=& 0 \end{array} \right.
$$
It is immediately checked that
$S_1 S_2 S_1 = S_2 S_1S_2 = S_1^3 = S_2^3 = 0$,
and that $S_1^2 S_2^2. w_r = w_{r+2}$ for all $r \geq 1$. This proves that $S_1,S_2$ defines
on $\mathcal{E}$ a structure of module over the algebra $A$ that we are considering. 
%(this module
%can a posteriori be identified with a quotient of  $w_r \equiv s_1^2 s_2^2 s_1^2\dots $ ($r$ terms in the letters $s_i^2$)
If $A$ were
finitely generated as a $k$-module, $A.w_1 \subset E$ would also be finitely generated as a $k$-module,
contradicting the fact that it contains an infinite subset of a basis for $E$.
\end{proof}
\begin{corollary} Let $m \in \Z \setminus \{ -1,1 \}$. The unital $\Z$-algebra defined by generators
$s_1,s_2$ and relations $s_1 s_2 s_1 = s_2 s_1 s_2$, $s_1^3 = m$, $s_2^3 = m$ is not
finitely generated as a $\Z$-module.
\end{corollary}
\begin{proof}
Choosing a prime $p$ dividing $m$, we get that this $\Z$-algebra admits for quotient the algebra
defined in the proposition for $k = \F_p$.
\end{proof}

%\subsection{An explicit torsion element in type $G_4$}

An immediate corollary is that we cannot expect the BMR conjecture
to hold without invertibility conditions.
%of the coefficient of $1$last parameter.
%when $c$ is not invertible. 
This is a big difference with the
Coxeter case. More precisely we prove the following.
\begin{proposition} \label{propG4}
The algebra defined over $\Z[a,b,c]$ by generators $s_1,s_2$ and relations $s_1 s_2 s_1 = s_2 s_1 s_2$
% The quotient of the group algebra of the braid group of $G_4$
%over $\Z[a,b,c]$ by the relations 
and $s_i^3 = a s_i^2 + b s_i + c$ for $i \in \{1,2 \}$
is not finitely generated as a $\Z[a,b,c]$-module. %and has non-trivial torsion. 
In the specialization
$a=b=0$, a non-zero torsion element
of the corresponding $\Z[c]$-module is provided by $(s_1^2 s_2^2)^6 - c^8$.
\end{proposition} 
\begin{proof}
Infinite generation follows again from the specialisation $a=b=c=0$.
By the computation described in figure \ref{figtorsion}, we prove that $c( (s_1^2 s_2^2)^6 - c^8) = 0$. 
Specializing to $a=b=c = 0$ we get that, on the $\Z$-module already used above, $s_1^2s_2^2$
is mapped to an endomorphism of infinite order, thus proving $(s_1^2 s_2^2)^6 \neq c^8$.
%Moreover,
%itis easily checked that $(s_1^2 s_2^2)^6$ acts by $c^8$ in every
%since the BMR conjecture is known for $G_4$, we can check that $s_1^2 s_2^2$ is eq
%it is easily checked that 
%\dots 

\end{proof}

\begin{figure}
\fbox{
$
\begin{array}{lcl}
c (s_1^2 s_2^2)^6 &=& c s_1^2 s_2^2s_1^2 s_2^2s_1^2 s_2^2s_1^2 s_2^2s_1^2 s_2^2s_1^2 s_2^2 \\
 &=&  s_1c s_1 s_2^2s_1^2 s_2^2s_1^2 s_2^2s_1^2 s_2^2s_1^2 s_2^2s_1^2 s_2^2 \\
 &=&  s_1s_2^3 s_1 s_2^2s_1^2 s_2^2s_1^2 s_2^2s_1^2 s_2^2s_1^2 s_2^2s_1^2 s_2^2 \\
 &=&  s_1s_2^2 (s_2 s_1 s_2) s_2s_1^2 s_2^2s_1^2 s_2^2s_1^2 s_2^2s_1^2 s_2^2s_1^2 s_2^2 \\
 &=&  s_1s_2^2 s_1 (s_2 s_1 s_2) s_1^2 s_2^2s_1^2 s_2^2s_1^2 s_2^2s_1^2 s_2^2s_1^2 s_2^2 \\
 &=&  s_1s_2^2 s_1 s_1 s_2 (s_1 s_1^2) s_2^2s_1^2 s_2^2s_1^2 s_2^2s_1^2 s_2^2s_1^2 s_2^2 \\
 &=&  c s_1s_2^2 s_1 s_1 (s_2  s_2^2)s_1^2 s_2^2s_1^2 s_2^2s_1^2 s_2^2s_1^2 s_2^2 \\
 &=&  c^2 s_1s_2^2 s_1 (s_1 s_1^2) s_2^2s_1^2 s_2^2s_1^2 s_2^2s_1^2 s_2^2 \\
 &=&  c^3 s_1s_2^2 s_1  s_2^2s_1^2 s_2^2s_1^2 s_2^2s_1^2 s_2^2 \\
 &=&  c^3 s_1s_2 (s_2 s_1 s_2) s_2s_1^2 s_2^2s_1^2 s_2^2s_1^2 s_2^2 \\
 &=&  c^3 s_1s_2 s_1 (s_2 s_1 s_2)s_1^2 s_2^2s_1^2 s_2^2s_1^2 s_2^2 \\
 &=&  c^3 s_1s_2 s_1 s_1 s_2 (s_1s_1^2) s_2^2s_1^2 s_2^2s_1^2 s_2^2 \\
 &=&  c^4 s_1s_2 s_1 s_1 (s_2  s_2^2)s_1^2 s_2^2s_1^2 s_2^2 \\
 &=&  c^5 s_1s_2 s_1 (s_1 s_1^2) s_2^2s_1^2 s_2^2 \\
 &=&  c^6 (s_1s_2 s_1)  s_2^2s_1^2 s_2^2 \\
 &=&  c^6 s_2s_1 (s_2  s_2^2)s_1^2 s_2^2 \\
 &=&  c^7 s_2(s_1 s_1^2) s_2^2 \\
 &=&  c^8 s_2 s_2^2 \\
 &=&  c^9  \\
\end{array}
$
}
\caption{Torsion element in type $G_4$ : $c( (s_1^2 s_2^2)^6 - c^8) = 0$}
\label{figtorsion}
\end{figure}

\subsubsection{The case of $G_{12}$}

The example of $G_4$ might suggest that differences with the Coxeter case may happen only when the reflections have order more than $2$.
We prove that this is not the case, by considering the reflection group of type $G_{12}$, whose reflections all
have order $2$.
A suitable monoid for its braid group is given by
the presentation $\langle A,B,C \ | \ ABCA = BCAB = CABC \rangle$. The generators
are braided reflections, and the monoid is known to be Garside (see \cite{PICANTIN}).

\begin{proposition}
 Let $k$ be a ring. The unital $k$-algebra defined by generators
$A,B,C$ and relations $ABCA = BCAB = CABC$, $A^2=B^2=C^2= 0$ is not
finitely generated as a $k$-module. The same holds if the latter relations
are replaced by $A^2 =A$, $B^2 = B$, $C^2 = C$.
\end{proposition}
\begin{proof}
We introduce the free modules $\mathcal{W}^+, \mathcal{W}^-$,
with bases $w_r^+, w_r^-$, for $r \geq 1$, and
make $A,B,C$ act through $C . w_r^+ = C. w_r^- = 0$
and 
$$
\left\lbrace \begin{array}{lcl}
A. w_r^+ &=& 0 \\
A. w_r^- &=& w_{r+1}^+ \\
\end{array}
\right.
\left\lbrace \begin{array}{lcl}
B w_r^+ &=& w_{r+1}^- \\
B w_r^- &=& 0
\end{array}
\right.
$$
One easily gets that $A^2$, $B^2$ et $C^2$ act
by $0$, as well as $ABCA$, $BCAB$, $CABC$, thus defining
a module structure for the first algebra. Since one can
check that $AB$ acts by $w_r^+ \mapsto w_{r+2}^+$
one gets the conclusion. For the
second algebra, we make still $C$ act by $0$,
whereas $A. w_r^+ = w_{r+1}^-$, $A.w_r^- = w_r^-$,
$B. w_r^- = w_{r+1}^+$, $B.w_r^+ = w_r^+$. This time
$BA$ maps  $w_r^+ \mapsto w_{r+2}^+$.
\end{proof}

\subsubsection{The case of $G(d,1,2)$}
We finally make a third example, this time inside the infinite series. The usual Hecke algebra
has a presentation with generators $t,s$ and relations $stst = tsts$, $t^d = a_0 + a_1 t + \dots + a_{d-1} t^{d-1}$,
$s^2 = \alpha s + \beta$, defined over $\Z[a_i, a_0^{-1}, \alpha,\beta, \beta^{-1}]$, and the BMR conjecture is known
for them, by work of Ariki and Koike \cite{ARIKIKOIKE} . However, and somewhat surprisingly in view of the previous examples,
it can be proved (see \cite{ARIKIKOIKE}) that it is actually finitely generated over $\Z[a_i, \alpha, \beta, \beta^{-1}]$. This feature
is true for the general case of the $G(d,1,r)$. For $r=2$,
an explicit spanning set of $2d^2 = |G(d,1,2)|$ elements is given by the $t^{m} u^n s^{\eps}$ for $0 \leq m ,n\leq d-1$
and $\eps \in \{ 0,1 \}$, for $u = sts$. The fact that it is a spanning set over $\Z[a_i, \alpha, \beta, \beta^{-1}]$
can be deduced from the easily checked relations  $tu = ut$, $us = \beta st + a u$, $st = \beta^{-1} us - \beta^{-1} a u$ ;
and their consequences $s.  u^{n+1} = \beta tsu^{n} + \alpha u^{n+1}$,
$st^{m+1} u^n = \beta^{-1} u. s t^m u^n - \beta^{-1} \alpha t^m u^{n+1} $. However,
$\beta$ really needs to be invertible, as we illustrate now.

\begin{proposition}
 Let $k$ be a ring. The unital $k$-algebra defined by generators
$t,s$ and relations $stst = tsts$, $t^3 = 0$, $s^2 = s$, is not
finitely generated as a $k$-module.
%The same holds if the latter relations
%are replaced by $a^2 = a$, $b^2 = b$, $c^2 = c$.
\end{proposition}
\begin{proof}
Let $E$ be the free $k$-module with basis the
elements $w_r, w'_r,y_r$ for $r \geq 1$. We make $s,t$
act on $E$ through
$$
\left\lbrace \begin{array}{lcl} s . w_r &=& w_r \\ t . w_r &=& y_{r} \end{array} \right.
\left\lbrace \begin{array}{lcl} s . w'_r  &=& w_{r+1} \\ t . w'_r &=& 0 \end{array} \right.
\left\lbrace \begin{array}{lcl} s . y_r &=& 0 \\ t . y_r &=& w'_r \end{array} \right.
%\left\lbrace \begin{array}{lcl} s_1 . y'_r &=& w_r \\ s_2 . y'_r &=& 0 \end{array} \right.
$$
One checks easily than $s^2$ acts like $s$ and that both $t^3$ and $stst=tsts$ act by $0$. The
corresponding module is generated by $w_1$. Since it is a free $k$-module of
infinite rank this proves that the algebra of the statement is not finitely generated.
\end{proof}

\begin{corollary} 
The algebra defined by
generators $t,s$ and relations $stst = tsts$, $t^d = a_0 + a_1 t + \dots + a_{d-1} t^{d-1}$,
$s^2 = \alpha s + \beta$, is not finitely generated over $\Z[a_i, \alpha,\beta]$ when $d \geq 3$.
\end{corollary}
\begin{proof}
The specialization of this algebra at $a_i = 0$, $\alpha = 1, \beta  = 0$ 
%of this algebra
admits a quotient (by the ideal generated by $t^3$) which is not a finitely
generated $\Z$-module, whence the conclusion.
\end{proof}

Note that the assumption $d \geq 3$ is necessary, because the case $d = 2$ corresponds to
a Coxeter group, for which $0$-Hecke algebras \emph{are} finitely generated.

Inside the infinite series, R. Rouquier communicated to us the following other
example of the group $G(4,2,2)$, for the presentation $\langle A,B,C \ | \ ABC = BCA = CAB, A^2 = B^2 = C^2 = 0 \rangle$.
Then, it can be checked that the algebra $\langle A,B | A^2 = B^2 = 0 \rangle$ naturally
embeds inside the Hecke algebra $H$, and that it is not finitely generated.

\section{The Hecke algebra of $G_{26}$}

According to \cite{JMUL}, the BMR conjecture has been checked to hold for $G_{26}$ by J. Müller, using Linton's algorithm of
vector enumeration (see \cite{LINTON}) and unpublished software.
For completeness, we also note that the Schur elements of $G_{26}$ have been computed in \cite{MALLE2}, \S 6C,
under the additional assumption of the existence of a suitable trace form.

\begin{theorem} \label{theoG26} The BMR conjecture holds for $G_{26}$.
\end{theorem}

We recall that the group $G_{26}$ is a Shephard group, and a quotient of the braid group of type $B_3$. It is the largest
of the two `linearizations' of the group of automorphisms of the projective `Hessian configuration' (see e.g. \cite{ORLIKTERAO},
example 6.30 p. 226),
the other one being $G_{25}$.

We take for generators of the braid group $B$ of $G_{26}$ the elements $t, s_2, s_1$ satisfying
the braid relations $t s_2 t s_2 = s_2 t s_2 t$, $s_2 s_1 s_2 = s_1 s_2 s_1$ and $t s_1 = s_1 t$.

The generic Hecke algebra $\hat{A}$ of $G_{26}$ is then defined over the
ring  $R = \Z[a,b,c^{-1},d,e^{-1}]$, with generators $s_1, s_2,t$ subject
to the above braid relations, and in addition to the relations $s_i^3 = a s_1^2 + b s_1 + c$  and $t^2 = dt + e$.
The ring $\hat{A}$ admits useful (skew)automorphisms,
 defined by  $s_i \mapsto s_i^{-1}$,
$t \mapsto t^{-1}$, $a \mapsto - b c^{-1}$, $b \mapsto - a c^{-1}$, $c \mapsto c^{-1}$,
$d \mapsto -d e^{-1}$, $e \mapsto e^{-1}$. We let $\phi$ denote the automorphism,
$\psi$ the corresponding skew-automorphism.

Let $A_3 = \langle s_1,s_2 \rangle \subset \hat{A}$ and $\hat{A}^{(n+1)} = \hat{A}^{(n)} v A_3$. For technical
reasons, we also introduce the following intermediate bimodules
$$
\begin{array}{lcl}
\hat{A}^{(2 \frac{1}{2})} &= & \hat{A}^{(2)} + A_3 t s_2 s_1 t s_2 t A_3
+ A_3 t s_2 s_1 t s_2^{-1} t A_3$$ {} $$
+ A_3 t s_2 s_1^{-1} t s_2^{-1} t A_3
+ A_3 t s_2^{-1} s_1^{-1} t s_2^{-1} t A_3 \\
\hat{B} &=&  \AAA + A_3 C^2 + A_3 C^{-2} \\
\end{array}
$$

We have
the following inclusion/equalities, some of them being
obvious from the definitions, the other ones being proved in the sequel.

\xymatrix{
 & & & \hat{A}^{(3)} \ar@{^(->}[dr]& \\
\hat{A}^{(1)} \ar@{^(->}[r] & \hat{A}^{(2)} \ar@{^(->}[r] & \AAA \ar@{^(->}[ur] \ar@{^(->}[dr]  & & \ \ \ \ \ \ \ \ \ \ \ \ \hat{A}^{(4)}= \hat{A}^{(5)} = \hat{A} \\
 & & & \hat{B}\  \ar@{^(->}[ur] & \\
}

Let $C = (t s_2 s_1)^3$. It is central in $\hat{A}$, as it generates the center of the braid group, and its image in $G_{26}$
has order $6$. We let $u_i = R + R s_i + R s_i^{-1}$ denote the subalgebra generated by $s_i$, and $v = R + R t$
the subalgebra generated by $t$. We will need the following results on the `parabolic' subalgebras
$A_3 = \langle s_1, s_2 \rangle$ and $\langle s_2, t \rangle$, which correspond to the rank 2 parabolic subgroups
of Shephard-Todd type $G_4$ and $G(3,1,2)$, respectively.

\begin{proposition} {\ } \label{lemG26L5}
\begin{enumerate}
\item $\langle s_1,s_2 \rangle =  u_1 u_2 u_1 + u_1 s_2 s_1^{-1} s_2 $
\item $\langle s_1,s_2 \rangle =  u_1 + u_1 s_2 u_1 + u_1 s_2^{-1} u_1 + u_1 s_2 s_1^{-1} s_2 $
\item $\langle s_2,t \rangle = \sum_{a \in \{-1,0,1 \}} R s_2^a +\sum_{a,b \in \{-1,0,1 \}} R s_2^a t s_2^b + \sum_{a \in \{-1,0,1 \}} R s_2^a t s_2 t
+ \sum_{a \in \{-1,0,1 \}} R s_2^a t s_2^{-1} t
$
\end{enumerate}
\end{proposition}
\begin{proof} (1) and (2) are easy and proved in \cite{conjcub}.
We prove (2). The RHS clearly contains $1$ and is stable under left multiplication by $s_2$. It thus sufficient to prove
that it is stable under left multiplication by $t$. Let $U$ denote the RHS. Since $\langle s_2 \rangle$ is $R$-spanned
by $1,s_2,s_2^2$, we need to prove 
$t s_2^{\alpha} t s_2^b  \in U$ and $t s_2^{\alpha} t s_2^{\beta} t  \in U$ for all $b \in \{ 0,1,2 \}$
and $\alpha,\beta \in \{ -1,1 \}$. If $\alpha = 1$ we have
$t s_2^{\alpha} t s_2^b = ts_2 t s_2^b$. If $b = 1$ we get in addition $ts_2 t s_2^b = ts_2 t s_2 = s_2 t s_2 t \in U$;
If $b = 1$ we get $(ts_2 t s_2) s_2 = s_2 (t s_2t s_2) = s_2^2 t s_2 t \in U$; if $b=0$ we know $t s_2 t \in U$.
If $\alpha = -1$ we have $t s_2^{-1} t s_2^b \in R^{\times} t^{-1}s_2^{-1} t^{-1} s_2^b + \langle s_2 \rangle t \langle s_2 \rangle
\subset R^{\times} t^{-1}s_2^{-1} t^{-1} s_2^b + U$
and the proof of $t^{-1}s_2^{-1} t^{-1} s_2^b \in U$ is similar, taking this time $b \in \{-2,-1,0\}$ and using $t^{-1} s_2^{-1}t^{-1} s_2^{-1}
=  s_2^{-1}t^{-1} s_2^{-1}t^{-1} $ instead. This in particular implies that
$t (s_2^{\alpha} t s_2^{\beta} t) = (t s_2^{\alpha} t s_2^{\beta}) t \in U t $.
The same proof, reading from the right, proves that $s_2^{\alpha} t s_2^b t \in U$ for all $\alpha, b$,
which clearly implies $U t \subset U$, and this concludes the proof.

\end{proof}

\subsection{Bimodule decompositions of $\hat{A}^{(k)}$, $1 \leq k \leq 3$}
%\begin{lemma}{\ } \label{lemG26L1}
\begin{proposition}{\ } \label{lemG26L1} (bimodule decomposition of $\hat{A}^{(1)}$ and $\hat{A}^{(2)}$)
\begin{enumerate}
\item $\hat{A}^{(1)} = A_3 + A_3 t A_3$
\item $\hat{A}^{(2)} = \hat{A}^{(1)} + A_3 ts_2t A_3+A_3 ts_2^{-1}t A_3+ A_3 ts_2 s_1^{-1}s_2 t A_3$
\end{enumerate}
%\end{lemma}
\end{proposition}
\begin{proof}
(1) is clear, as $v=\langle t \rangle$ is $R$-generated by $1$ and $t$. For proving (2)
we note that $A_3 = u_1 s_2 s_1^{-1} s_2 + u_1 u_2 u_1$ hence
$tA_3 t \subset tu_1 s_2 s_1^{-1} s_2t + tu_1 u_2 u_1t
\subset u_1 t s_2 s_1^{-1} s_2t + u_1 tu_2 tu_1
\subset u_1 t s_2 s_1^{-1} s_2t + u_1 t^2 u_1 +  u_1 ts_2 tu_1+ u_1 ts_2^{-1} tu_1$. This proves (2).
\end{proof}

\begin{lemma} {\ } \label{lemG26L2}
\begin{enumerate}
\item For all $i,j \in \{1,2 \}$, $tu_i t u_j t \subset \hat{A}^{(2)}$
\item $tu_2 u_1 t u_2 t \subset \sum_{\alpha,\beta,\gamma} R t s_2^{\alpha} t s_1^{\beta} s_2^{\gamma} t$
\end{enumerate}
\end{lemma}
\begin{proof}
We prove (1). If $i=1$ or $j=1$ this is clear by the commutation relations. One can thus
assume $i=j=2$, and consider $ts_2^a t s_2^b t$ with $a,b \in \{0,1,2 \}$ since $u_2$
is $R$-spanned by $1,s_2$ and $s_2^2$. If $a = 0$
or $b=0$ this is clear. If $a=b=1$ then this is $(t s_2 t s_2) t = s_2 t s_2t^2 \in \hat{A}^{(2)}$ ; 
if $a=1$ and $ b=2$, then this is $(t s_2t s_2) s_2 t =  s_2(t s_2t s_2) t=  s_2^2 t s_2t^2  \in \hat{A}^{(2)}$ ;
the case $a=2$ and $b = 1$ is similar. We thus only
need to consider the case $a=b=2$. Using $s_2^2 \in R s_2^{-1} + R s_2 + R$
we get from the preceding cases $t s_2^2t s_2^2t \in t s_2^{-1}t s_2^{-1}t + \hat{A}^{(2)}$ ;
moreover $t \in R^{\times}t^{-1} +R$ hence  $ t s_2^{-1}t s_2^{-1}t  \in R^{\times} t^{-1} s_2^{-1}t^{-1} s_2^{-1}t^{-1}
+ \hat{A}^{(2)}$. Now  $(t^{-1} s_2^{-1}t^{-1} s_2^{-1})t^{-1} =  s_2^{-1}t^{-1} s_2^{-1}t^{-2} \in \hat{A}^{(2)}$ and this
concludes the proof of (1). (2) obviously follows from (1), as $u_i$ is the $R$-linear span of $1,s_i$ and $s_i^{-1}$.
\end{proof}

\begin{lemma} {\ } \label{lemG26L3}
\begin{enumerate}
\item $t s_2 s_1^{-1} t s_2^{-1} t \in A_3^{\times} t s_2^{-1} s_1 t s_2^{-1} t A_3^{\times}  + \hat{A}^{(2)}$
\item $t s_2^{-1} s_1 t s_2^{-1} t \in A_3^{\times} t s_2^{-1} s_1^{-1} t s_2 t  A_3^{\times} + \hat{A}^{(2)}$
\item $ts_2 s_1 t s_2^{-1} t \in A_3^{\times} ts_2   s_1^{-1} t s_2t A_3^{\times}$
\item $ t s_2 s_1^{-1} t s_2 t \in A_3^{\times} t s_2^{-1} s_1 ts_2t A_3^{\times} $
\end{enumerate}
\end{lemma}
\begin{proof}
We have $$
\begin{array}{lcl}
s_1(t s_2 s_1^{-1} t s_2^{-1} t) 
&=& t (s_1s_2 s_1^{-1}) t s_2^{-1} t \\
&=& t s_2^{-1} s_1 s_2 (t s_2^{-1} t s_2^{-1}) s_2\\
&\in& R^{\times} t s_2^{-1} s_1 s_2 (t^{-1} s_2^{-1} t^{-1} s_2^{-1}) s_2 + \hat{A}^{(2)}
\end{array}$$
and 
$t s_2^{-1} s_1 s_2 (t^{-1} s_2^{-1} t^{-1} s_2^{-1})
= t s_2^{-1} s_1 s_2  s_2^{-1} t^{-1} s_2^{-1}t^{-1}
= t s_2^{-1} s_1 t^{-1} s_2^{-1}t^{-1}
\in R^{\times} t s_2^{-1} s_1 t s_2^{-1}t + \hat{A}^{(2)}$, which
proves (1). Now 
$$\begin{array}{lcl}
(t s_2^{-1} s_1 t s_2^{-1}t) s_1^{-1} &=& 
t s_2^{-1}  t(s_1 s_2^{-1} s_1^{-1})t \\ &=& 
t s_2^{-1}  ts_2^{-1} s_1^{-1} s_2t \\
&\in& R^{\times} t^{-1} s_2^{-1}  t^{-1} s_2^{-1} s_1^{-1} s_2t  + \hat{A}^{(2)} \end{array} $$
and $ (t^{-1} s_2^{-1}  t^{-1} s_2^{-1}) s_1^{-1} s_2t
 =  s_2^{-1}  t^{-1} s_2^{-1} t^{-1} s_1^{-1} s_2t
 =  s_2^{-1}  t^{-1} s_2^{-1} s_1^{-1}t^{-1}  s_2t
 \in R^{\times} s_2^{-1}  t s_2^{-1} s_1^{-1}t  s_2t + \hat{A}^{(2)}$ and
 this proves (2). 
 
 We have
 $$
 \begin{array}{clclcl}
 & (ts_2 s_1 t s_2^{-1} t) s_1^{-1} 
& =& ts_2  t(s_1 s_2^{-1} s_1^{-1})t 
 &=& ts_2 ts_2^{-1} s_1^{-1} s_2t  \\
 =& s_2^{-1} (s_2  ts_2 t)s_2^{-1} s_1^{-1} s_2t 
& =&s_2^{-1} ts_2  ts_2 s_2^{-1} s_1^{-1} s_2t &
 =&s_2^{-1} ts_2   s_1^{-1} t s_2t
\end{array}
$$ 
and this proves (3). Now $s_1 t s_2 s_1^{-1} t s_2 t = 
 t (s_1 s_2 s_1^{-1}) t s_2 t = 
 t s_2^{-1} s_1 (s_2t s_2 t) = 
 t s_2^{-1} s_1 ts_2t s_2  = 
 (t s_2^{-1} s_1 ts_2t) s_2 $ and this proves (4).
% (5) is then an immediate consequence of (1-4) and of the previous lemma.
\end{proof}
%\begin{lemma} \label{lemG26L4}
\begin{proposition} \label{propG26L4} (bimodule decomposition
of $\hat{A}^{(3)}$)
$$\hat{A}^{(3)} = \hat{A}^{(2)} + A_3 t s_2 s_1^{-1} s_2 ts_2 s_1^{-1} s_2 t A_3 + A_3 t s_2 s_1 t s_2 t A_3
+ A_3 t s_2 s_1 t s_2^{-1} t A_3$$ {} $$
+ A_3 t s_2 s_1^{-1} t s_2^{-1} t A_3
+ A_3 t s_2^{-1} s_1^{-1} t s_2^{-1} t A_3
$$
\end{proposition}
%\end{lemma}
\begin{proof}
Since $A_3 = u_1 s_2 s_1^{-1} s_2 + u_1 u_2 u_1 = s_2 s_1^{-1} s_2 u_1 + u_1 u_2 u_1$
we get 
$$tA_3 t A_3 t \subset A_3 ts_2 s_1^{-1} s_2ts_2 s_1^{-1} s_2t A_3 + A_3 t u_2  t A_3 t A_3 + A_3 t A_3  t u_2 t A_3.
$$
Now $t u_2  t A_3 t = R t^2 A_3 t + R t s_2 t A_3 t + R t s_2^{-1} t A_3 \subset \hat{A}^{(2)} + R t s_2 t A_3 t + R t s_2^{-1} t A_3$.
Using again $A_3 =  s_2 s_1^{-1} s_2 u_1 + u_1 u_2 u_1$ we get 
$t s_2 t A_3 t \subset (t s_2 ts_2) s_1^{-1} s_2 u_1t +t s_2 t u_1 u_2 u_1t
\subset  s_2 ts_2 ts_1^{-1} s_2 tu_1 +t s_2 t u_1 u_2 tu_1  \subset A_3 t u_2 u_1 t u_2 t A_3$ ;
using $A_3 =  s_2^{-1} s_1 s_2^{-1} u_1 + u_1 u_2 u_1$
we get similarly
$t s_2^{-1} t A_3 t 
\subset (t s_2^{-1} ts_2^{-1}) s_1 s_2^{-1} u_1t + t s_2^{-1} tu_1 u_2 u_1t
\subset  s_2^{-1} ts_2^{-1} ts_1 s_2^{-1} tu_1 + t s_2^{-1} tu_1 u_2 tu_1
\subset A_3 t u_2 u_1 t u_2 t A_3$.
This yields $t u_2  t A_3 t \subset \hat{A}^{(2)} + A_3 t u_2 u_1 t u_2 t A_3$. In a similar way, we leave to
the reader to check that
$t A_3 t u_2  t \subset \hat{A}^{(2)} + A_3 t u_2 u_1 t u_2 t A_3$.
%can prove that
This implies 
$$tA_3 t A_3 t \subset \hat{A}^{(2)} + A_3 t u_2 u_1 t u_2 t A_3 + A_3 ts_2 s_1^{-1} s_2ts_2 s_1^{-1} s_2t A_3 .
$$
The conclusion then follows from lemmas \ref{lemG26L2} and \ref{lemG26L3}.
\end{proof}

\subsection{The bimodule $\AAA$}

\begin{proposition} {\ } 
\begin{enumerate}
\item $C \in t s_2 s_1t s_2 t A_3^{\times}$ and $C^{-1} \in A_3^{\times} t s_2^{-1} s_1^{-1}t s_2^{-1} t + \hat{A}^{(2)}$
\item $\hat{A}^{(3)} = \hat{A}^{(2 \frac{1}{2})} + 
A_3 t s_2 s_1^{-1} s_2 ts_2 s_1^{-1} s_2 t A_3 $
\item $\hat{A}^{(2 \frac{1}{2})} = \hat{A}^{(2)} + A_3 t s_2 s_1 t s_2 t
+ A_3 t s_2 s_1 t s_2^{-1} t A_3
+ A_3 t s_2 s_1^{-1} t s_2^{-1} t A_3
+ A_3 t s_2^{-1} s_1^{-1} t s_2^{-1} t$
\end{enumerate}
\end{proposition}
\begin{proof} 
We have $C = t s_2 s_1t s_2 s_1t s_2 s_1 = 
t s_2 s_1t s_2 t s_1s_2 s_1  \in t s_2 s_1t s_2 t A_3^{\times}$. 
One gets similarly $C^{-1} \in A_3^{\times} t^{-1} s_2^{-1} s_1^{-1}t^{-1} s_2^{-1} t^{-1}$.
Since $t^{-1} \in R^{\times} t +R$ this implies  $C^{-1} \in A_3^{\times} t s_2^{-1} s_1^{-1}t s_2^{-1} t + \hat{A}^{(2)}$. This proves (1).
(2) follows from proposition \ref{propG26L4}. Since $C$ is central,
(3) then follows from (1).
\end{proof}

%We need the following description of the parabolic subalgebra $\langle s_2, t \rangle$,
%corresponding to the `standard' parabolic subgroup of type $G(3,1,2)$.

%\subsection{Powers of $C$ -- preliminary lemmas}

%\subsection{Conclusion}

%\subsubsection{Preliminary lemmas}
%D'autre part, on montre le lemme suivant.

We now compute the number of elements which are needed to generate $\AAA$ modulo $\hat{A}^{(2)}$
as a $A_3$-module. We need the following two lemmas.

\begin{lemma} \label{lemA3cmod} {\ }
\begin{enumerate}
\item For all $\alpha \in \{0,1, -1 \}$, $(t s_2 s_1) s_2^{\alpha} = s_1^{\alpha} (ts_2s_1)$
\item $ts_2 s_1 t s_2^{\pm 1} t u_2 \subset \hat{A}^{(2)} + A_2 t s_2 s_1 t s_2 t + A_2 t s_2 s_1 t s_2^{-1} t$  
\item $$t s_2 s_1 t s_2^{\pm 1} t A_3 \subset  \hat{A}^{(2)} + \sum_{a \in \{ -1,0,1 \}} \sum_{b,\eps \in \{ -1,1 \}} A_2 t s_2 s_1 t s_2^{\eps} t s_1^b s_2^a
%$${} \\ $$
 +\sum_{\eps \in \{ -1,1 \}} \left( A_2 t s_2 s_1 t s_2^{\eps} t s_1 s_2^{-1} s_1 + A_2 t s_2 s_1 t s_2^{\eps} t \right)$$
 \item $$
 ts_2 s_1 t s_2^{\pm 1} t A_3
 \subset \hat{A}^{(2)} + A_3 t s_2 s_1 t s_2 t + 
 \sum_{a \in \{ -1,0,1 \}} \sum_{b \in \{ -1,1 \}} A_2 t s_2 s_1 t s_2^{-1} t s_1^b s_2^a
$${} \\ $$
 + \left( A_3 t s_2 s_1 t s_2^{-1} t s_1 s_2^{-1} s_1 + A_2 t s_2 s_1 t s_2^{-1} t \right)$$
 \item $ \hat{A}^{(2)} + \sum_{\eps \in \{-1,1 \} } A_3 ts_2 s_1 t s_2^{\pm 1} t A_3$
 is spanned as a $A_3$-module by $\hat{A}^{(2)} $ and $9$ elements originating from the braid group. 
\end{enumerate}
\end{lemma}
\begin{proof}
For (1), this is because $ts_2 s_1 s_2^{\alpha} = t(s_2 s_1 s_2^{\alpha} ) = ts_1^{\alpha} s_2 s_1 = s_1^{\alpha}  ts_2 s_1 $.
Since $t s_2^{\pm 1} t u_2 \subset \langle s_2, t \rangle$, proposition \ref{lemG26L5} implies
$$
\begin{array}{lcl}
ts_2 s_1 t s_2^{\pm 1} t u_2 & \subset & \hat{A}^{(2)} + \sum_{a \in \{ -1,0,1 \}} R t s_2 s_1 s_2^a t s_2 t + \sum_{a \in \{-1,0,1 \}} R t s_2 s_1 s_2^a t s_2^{-1} t \\
& & \hat{A}^{(2)} + \sum_{a \in \{ -1,0,1 \}} R s_1^at s_2 s_1  t s_2 t + \sum_{a \in \{-1,0,1 \}} R s_1^at s_2 s_1  t s_2^{-1} t \\
& & \hat{A}^{(2)} + A_2 t s_2 s_1  t s_2 t + A_2 t s_2 s_1  t s_2^{-1} t \\
\end{array}
$$
that is (2). Then (3) is an immediate consequence of (2) and of the decomposition of
$A_3$ as $\langle s_2 \rangle$-module given by proposition \ref{lemG26L5}
up to exchanging $s_1$ and $s_2$ (see also
\cite{conjcub}). (4) is readily
deduced because $t s_2 s_1 t s_2 t$ commutes with $A_3$, and then (5)
is clear.
\end{proof}

%(voir plus loin pour la définition de $\Phi$)
\begin{lemma} \label{lemA3cmodPhi}
The image under $\phi$ of $\hat{A}^{(2)} + A_3 t s_2^{-1}
s_1^{-1} t s_2^{-1} t + A_3 t s_2 s_1^{-1} t s_2^{-1} t A_3$
is $\hat{A}^{(2)} + A_3 t s_2 s_1 t s_2 t + A_3 t s_2 s_1 t s_2^{-1} t A_3$. Thus
$$
\AAA = \hat{A}^{(2)} + A_3 t s_2 s_1 t s_2 t + A_3 t s_2 s_1 t s_2^{-1} t A_3
+ A_3 \phi(t s_2 s_1 t s_2 t) + A_3 \phi(t s_2 s_1 t s_2^{-1} t) A_3
$$
\end{lemma}
\begin{proof}
This image is clearly
$\hat{A}^{(2)} + A_3 t^{-1} s_2 s_1 t^{-1}s_2 ^{-1}t^{-1} + A_3 t^{-1} s_2^{-1} s_1
t^{-1}s_2t^{-1}A_3$,
that is
$\hat{A}^{(2)} + A_3 t s_2 s_1 ts_2 ^{-1} t+ A_3 t s_2^{-1} s_1
ts_2tA_3$ by $t^{-1} \in R^{\times} t + R$,
hence
$\hat{A}^{(2)} + A_3 t s_2 s_1 ts_2 ^{-1} t+ A_3 t s_2 s_1^{-1}
ts_2tA_3$ because
$s_1^{-1}(ts_2^{-1} s_1 t s_2 t) = 
t(s_1^{-1}s_2^{-1} s_1) t s_2 t) = 
ts_2s_1^{-1} s_2^{-1} (t s_2 t) = 
(ts_2s_1^{-1}  t s_2 t)s_2^{-1}$. Now $ t s_2 s_1^{-1}
ts_2t \in A_3^{\times} t s_2 s_1 t s_2^{-1} t A_3^{\times}$
by lemma \ref{lemG26L3} %(lemme 4 - 2011) 
(3), and this concludes the proof of the lemma,
the last equality being an obvious consequence.

\end{proof}

These two lemmas imply the following proposition.

\begin{proposition} \label{propX}
As a $A_3$-module, $\AAA$ is generated by $\hat{A}^{(2)}$ together with $2 \times 9 = 18$ elements
originating from the braid group.
\end{proposition}

Additional properties of $\AAA$ include the following two results.

%Pour faire les calculs modulo $\hat{A}^{(2 \frac{1}{2})}$, on utilise
\begin{lemma} \label{lemtuuutuut} Whatever the choices of signs $\pm$,
\begin{enumerate} 
%\item $t^{\pm} u_2 u_1 u_2 t^{\pm}  u_1 u_2 t^{\pm}  \subset \hat{A}^{(2 \frac{1}{2})}$
\item $t^{\pm} u_2 u_1 u_2 t^{\pm}  u_1 u_2 t^{\pm}  \subset \hat{A}^{(2 \frac{1}{2})}$
\item $t^{\pm} u_2 u_1  t^{\pm}  u_2 u_1 u_2 t^{\pm}  \subset \hat{A}^{(2 \frac{1}{2})}$
\end{enumerate}
\end{lemma}
\begin{proof}
(1). Since $t^{-1} \in R t + R$ and $\hat{A}^{(2)} \subset \hat{A}^{(2 \frac{1}{2})}$, it suffices to show
$tu_2 u_1 u_2 t u_1 u_2 t \subset \hat{A}^{(2 \frac{1}{2})}$.
Now $tu_2 u_1 u_2 t u_1 u_2 t = tu_2 u_1 u_2 u_1t  u_2 t$ and
$u_2 u_1 u_2 u_1 = A_3 = u_1 u_2 u_1 u_2$, thus
$tu_2 u_1 u_2 u_1t  u_2 t = tu_1u_2 u_1 u_2 t  u_2 t
=u_1 tu_2 u_1 u_2 t  u_2 t$. Now $u_2 t  u_2 t \subset \langle s_2 ,t \rangle$ hence,
by proposition \ref{lemG26L5} %6 (2011)
(and applying the skew-automorphism induced by $s_2 \mapsto s_2^{-1}, t \mapsto t^{-1}$)
we have $u_2 t  u_2 t \subset u_2 + u_2 t u_2 + t s_2 t u_2 + t s_2^{-1} t u_2$
whence $tu_2 u_1 u_2 t  u_2 t \subset \hat{A}^{(2)} +  t  u_2 u_1 t s_2 t u_2 +t  u_2 u_1  t s_2^{-1} t u_2
\subset  \hat{A}^{(2 \frac{1}{2})}$ by lemmas \ref{lemG26L2} %3(2011) 
(2) and lemma \ref{lemG26L3}. %4 (2011).
This proves (1), and (2) follows from (1) under application of the skew-automorphism already mentionned.

\end{proof}

\begin{proposition} \label{lemstust} {\ }
%\begin{lemme} \label{lemstust} {\ }
\begin{enumerate}
\item $\AAA = \hat{A}^{(2)} +A_3  \langle t \rangle u_2 u_1 \langle t \rangle  u_2 \langle t \rangle A_3
 = \hat{A}^{(2)} +A_3 \langle t \rangle u_2  \langle t \rangle u_1 u_2 \langle t \rangle A_3$
\item $\AAA$ is stable under $\phi$ and $\psi$.
\item  $ \langle s_2 , t \rangle u_1 \langle s_2, t \rangle \subset \AAA$
\item $\langle t \rangle A_3 \langle s_2 ,t \rangle  \subset \AAA$
\item $\langle s_2 ,t \rangle  A_3 \langle t \rangle \subset \AAA$
\end{enumerate}
%\end{lemme}
\end{proposition}
\begin{proof}
(1) is an immediate consequence of the above, and (2) is a direct
consequence of (1).
Recall that (A) $\langle s_2, t \rangle = u_2 + u_2 t u_2 + u_2 ts_2 t + u_2 t s_2^{-1} t$ hence,
applying $\phi \circ \psi$, we have
(B) $\langle s_2, t \rangle = u_2 + u_2 t u_2 +  ts_2 tu_2 +  t s_2^{-1} tu_2$.
In particular, $\langle s_2,t \rangle \subset \hat{A}^{(2)} \subset \AAA$, thus it is sufficient
to show $\langle s_2,t \rangle s_1^{\alpha} \langle s_2,t \rangle \subset \AAA$ for $\alpha \in \{-1,1 \}$.
Since $\AAA$ is a $u_2$-bimodule, because of (A) this amounts to proving (a) $tu_2 s_1^{\alpha} \langle s_2,t \rangle
\subset \AAA$ and (b) $t s_2^{\eps} t s_1^{\alpha} \langle s_2,t \rangle \subset \AAA$ for all $\eps \in \{ -1, 1 \}$.
We start with (a). By (B), 
$$
t u_2 s_1^{\alpha} \langle s_2,t \rangle \subset \hat{A}^{(2)} + t u_2 s_1^{\alpha} t s_2 t u_2 + t u_2 s_1^{\alpha} t s_2^{-1} t u_2 \subset \AAA
$$
by lemma \ref{lemtuuutuut}, and this proves (a). We turn to (b). By (B),
$$
t s_2^{\eps} t s_1^{\alpha} \langle s_2,t \rangle \subset \hat{A}^{(2)} + t s_2^{\eps} t s_1^{\alpha} u_2 t u_2 + t s_2^{\eps} t s_1^{\alpha} t s_2 t u_2
+ t s_2^{\eps} t s_1^{\alpha} t s_2^{-1} t u_2.
$$
Now $t s_2^{\eps} t s_1^{\alpha} t s_2^{\pm 1} t u_2
= t s_2^{\eps}  s_1^{\alpha} t^2 s_2^{\pm 1} t u_2 \subset \AAA$
and $t s_2^{\eps} t s_1^{\alpha} u_2 t \subset \AAA$, and this concludes the proof of (3).
For proving (4), we use that $A_3 = u_1 u_2 u_1 u_2$, hence
$\langle t \rangle A_3 \langle s_2 ,t \rangle  \subset u_1 \langle t \rangle u_2 u_1 u_2 \langle s_2 ,t \rangle  \subset \AAA$
because of (3). Now (5) is a consequence of (4) by applying $\psi$.
\end{proof}

\subsection{Computation of $C^2$ modulo $\AAA$}

\begin{lemma} \label{lemC2modA2}
$C^2 \in A_3^{\times} t s_2 t s_1  s_2^{-1} s_1t s_2 t A_3^{\times} + \AAA.$
\end{lemma}
%
%$$
%C^2 \in \hat{A}^{(2)} + A_3^{\times} t s_2 t s_1  s_2^{-1} s_1t s_2 t(s_1 s_2 s_1)
%+ A_3 t s_2 s_1^{-1}  t   s_2 t (s_1^2 s_2 s_1) + A_3 t s_2 s_1 t s_2 t
%$$
%\end{lemma}
\begin{proof} We actually prove $C^2 \in \hat{A}^{(2)} + A_3^{\times} t s_2 t s_1  s_2^{-1} s_1t s_2 t(s_1 s_2 s_1)
+ A_3 t s_2 s_1^{-1}  t   s_2 t (s_1^2 s_2 s_1) + A_3 t s_2 s_1 t s_2 t$.
We have $C = (ts_2s_1)^3 \in A_3^{\times} t s_2 s_1 t s_2 t$, hence
$C^2 \in A_3^{\times} t s_2 s_1 t s_2 t t s_2 s_1t s_2 s_1t s_2 s_1$
and
$$t s_2 s_1 t s_2 t^2 s_2 s_1t s_2 s_1t s_2 s_1 \in R^{\times}
t s_2 s_1 t s_2^2 s_1t s_2 s_1t s_2 s_1
+
Rt s_2 s_1 t s_2 t s_2 s_1t s_2 s_1t s_2 s_1.
$$

$\bullet$ We have 
$$t s_2 s_1 t s_2^2 s_1t s_2 s_1t s_2 s_1 \in 
R^{\times}
t s_2 s_1 t s_2^{-1} s_1t s_2 s_1t s_2 s_1
+ R
t s_2 s_1 t s_2 s_1t s_2 s_1t s_2 s_1
+ R 
t s_2 s_1 t  s_1t s_2 s_1t s_2 s_1.
$$
We also have $t s_2 s_1 t s_2^{-1} s_1t s_2 s_1t s_2 s_1 = 
t s_2 t s_1  s_2^{-1} s_1t s_2 ts_1 s_2 s_1$,
and
$$
\begin{array}{lclcl}
t s_2 s_1 t s_2 s_1t s_2 s_1t s_2 s_1
&=& t s_2  t s_1 s_2 s_1t s_2 s_1t s_2 s_1
&=& t s_2  t s_2 s_1 s_2 t s_2 s_1t s_2 s_1 \\
&=& s_2 t s_2  t s_1 t s_2 t s_2 s_1 s_2 s_1
&=& s_2 t s_2   s_1 t^2 s_2 t s_2 s_1 s_2 s_1 \\
&\in & s_2 t s_2   s_1 t s_2 t s_2 s_1 s_2 s_1 &+& \hat{A}^{(2)} \\
&\subset & A_3 t s_2   s_1 t s_2 t A_3 & + & \hat{A}^{(2)} \\
& \subset &  A_3 t s_2   s_1 t s_2 t    &+ &\hat{A}^{(2)}. \\
\end{array}
$$
Finally, $t s_2 s_1 t  s_1t s_2 s_1t s_2 s_1
= t s_2 s_1^2  t^2 s_2 s_1t s_2 s_1
\in \hat{A}^{(2)} + R t s_2 s_1^2  t s_2 s_1t s_2 s_1$.
Moreover, $$t s_2 s_1^2  t s_2 s_1t s_2 s_1
\in R t s_2 s_1^{-1}  t s_2 s_1t s_2 s_1
+ R t s_2 s_1  t s_2 s_1t s_2 s_1
+ R t s_2  t s_2 s_1t s_2 s_1.$$
But since $t s_2  t s_2 s_1t s_2 s_1 = 
s_2 t s_2  t s_1t s_2 s_1
=
s_2 t s_2  s_1t^2  s_2 s_1 \in \hat{A}^{(2)}$
and $t s_2 s_1  t s_2 s_1t s_2 s_1
=
t s_2 s_1  t s_2 t s_1 s_2 s_1
\in t s_2 s_1  t s_2 t A_3 \subset
A_3 ts_2s_1ts_2t$, we get
$$
t s_2 s_1 t s_2^2 s_1t s_2 s_1t s_2 s_1  \in 
t s_2 t s_1  s_2^{-1} s_1t s_2 t(s_1 s_2 s_1)
+ A_3 t s_2   s_1 t s_2 t
+A_3 t s_2 s_1^{-1}  t s_2 t(s_1 s_2 s_1)
+
\hat{A}^{(2)}
$$

$\bullet$ We have
$$
\begin{array}{lclcl}
t s_2 s_1 t s_2 t s_2 s_1t s_2 s_1t s_2 s_1
&=& 
t s_2 s_1 s_2 t s_2 t   s_1t s_2 s_1t s_2 s_1
&=& 
t s_2 s_1 s_2 t s_2 t^2 s_1 s_2 s_1t s_2 s_1 \\
&=& 
t s_1 s_2 s_1 t s_2 t^2 s_1 s_2 s_1t s_2 s_1
&=& 
s_1t  s_2 s_1 t s_2 t^2 s_1 s_2 s_1t s_2 s_1
\end{array}
$$
and
$t  s_2 s_1 t s_2 t^2 s_1 s_2 s_1t s_2 s_1
\in 
R t s_2 s_1  t s_2 t s_1 s_2 s_1t s_2 s_1
+ 
R t s_2 s_1  t s_2  s_1 s_2 s_1t s_2 s_1$.
Moreover 
$$
\begin{array}{lclclcl}
t s_2 s_1  t s_2  s_1 s_2 s_1t s_2 s_1 &=& 
t s_2 s_1  t s_1  s_2 s_1 s_1t s_2 s_1 
&=&
t s_2 s_1^2  t   s_2 t s_1^2 s_2 s_1  \\
&\in&
R t s_2 s_1^{-1}  t   s_2 t (s_1^2 s_2 s_1) 
&+&
Rt s_2 s_1  t   s_2 t s_1^2 s_2 s_1 
&+&
t s_2   t   s_2 t s_1^2 s_2 s_1.
\end{array}
 $$
We have
$t s_2 s_1  t   s_2 t s_1^2 s_2 s_1  \in 
t s_2 s_1  t   s_2 t A_3
\subset A_3  t s_2 s_1  t   s_2 t + \hat{A}^{(2)}$,
$t s_2   t   s_2 t s_1^2 s_2 s_1= 
 s_2   t   s_2 t^2 s_1^2 s_2 s_1 \in \hat{A}^{(2)}$.

On the other hand,
$ t s_2 s_1  t s_2 t s_1 s_2 s_1t s_2 s_1
=  t s_2 s_1  t s_2 t s_2 s_1 s_2t s_2 s_1
=  t s_2 s_1   s_2 t s_2 t s_1 s_2t s_2 s_1
=  t s_1 s_2   s_1 t s_2 t s_1 s_2t s_2 s_1
=  s_1 t  s_2   s_1 t s_2 t s_1 s_2t s_2 s_1
=  s_1 t  s_2   s_1 t s_2 s_1 t  s_2t s_2 s_1
=  s_1 t  s_2   s_1 t s_2 s_1   s_2t s_2 t s_1
=  s_1 t  s_2   s_1 t s_1 s_2   s_1t s_2 t s_1
=  s_1 t  s_2   t s_1^2  s_2   s_1t s_2 t s_1
\in A_3t  s_2   t s_1^2  s_2   s_1t s_2 t A_3 $.
Moreover,
$t  s_2   t s_1^2  s_2   s_1t s_2 t
\in R t  s_2   t s_1^{-1}  s_2   s_1ts_2 t + 
R t  s_2   t s_1  s_2   s_1t s_2 t+ R t  s_2   t  s_2   s_1ts_2 t$
and $t  s_2   t  s_2   s_1ts_2 t =   s_2   t  s_2 t  s_1ts_2 t
=   s_2   t  s_2   s_1t^2s_2 t \in A_3 t s_2 s_1 t s_2 t + \hat{A}^{(2)}$,
$t  s_2   t s_1  s_2   s_1ts_2 t
=t  s_2   t s_2  s_1   s_2ts_2 t
=  s_2   t s_2 t s_1   s_2ts_2 t
=  s_2   t s_2 s_1  t^2   s_2ts_2 
\in A_3 t s_2 s_1  t   s_2t + \hat{A}^{(2)}$,
$t  s_2   t s_1^{-1}  s_2   s_1t s_2 t
=t  s_2   t s_2  s_1   s_2^{-1} ts_2 t
=  s_2   t s_2 t s_1    ts_2 t s_2^{-1}
=  s_2   t s_2  s_1    t^2 s_2 t s_2^{-1}
\in  A_3   t s_2  s_1    t s_2 t A_3
+ \hat{A}^{(2)}
\subset 
A_3   t s_2  s_1    t s_2 t 
+ \hat{A}^{(2)}
$.
This concludes the proof.
\end{proof}

%\subsection{Computations modulo $\AAA$}

For subsequent use, we also need the following related computation.

\begin{lemma} \ \label{lemSimpl0C3}
\begin{enumerate}
\item $ t s_2 s_1^{-1} t s_2 t s_1 C \in \AAA$
\item $ t s_2 s_1^{-1} t s_2 t s_1 t s_2 s_1 t s_2 t \in \AAA$
\end{enumerate}
\end{lemma}
\begin{proof}
(1) is a clear consequence of (2), so we focus on (2).
We have
$$
\begin{array}{clclcl}
& t s_2 s_1^{-1} t s_2 t s_1 t s_2 s_1 t s_2 t &=&
t s_2 s_1^{-1} t s_2 t^2 s_1  s_2 s_1 t s_2 t
&=&
t s_2  ts_1^{-1}  s_2s_1 t^2   s_2 s_1 t s_2 t \\
=&t s_2  ts_2  s_1s_2^{-1} t^2   s_2 s_1 t s_2 t &=& s_2  ts_2  t s_1s_2^{-1} t^2   s_2 s_1 t s_2 t s_2^{-1} s_2
&=& s_2  ts_2  t s_1s_2^{-1} t^2   s_2 s_1 s_2^{-1} t s_2 t  s_2 \\
= &s_2  ts_2  t s_1s_2^{-1} t^2   s_1^{-1} s_2 s_1 t s_2 t  s_2
&=& s_2  ts_2  t s_1s_2^{-1}  s_1^{-1}t^2   s_2 s_1 t s_2 t  s_2 
&=& s_2  ts_2  t s_2^{-1} s_1^{-1}  s_2t^2   s_2 s_1 t s_2 t  s_2 \\
= &s_2 s_2^{-1} ts_2  t  s_1^{-1}  s_2t^2   s_2 s_1 t s_2 t  s_2
&=&  ts_2  t  s_1^{-1}  s_2t^2   s_2 s_1 t s_2 t  s_2 \\
\in & R ts_2  t  s_1^{-1}  s_2t   s_2 s_1 t s_2 t  s_2
&+& R ts_2  t  s_1^{-1}  s_2   s_2 s_1 t s_2 t  s_2.
\end{array}
$$
Now, on the one hand
$ts_2  t  s_1^{-1}  s_2   s_2 s_1 t s_2 t 
=
ts_2  t  s_1^{-1}  s_2^2 s_1 t s_2 t 
=
ts_2  t  s_2  s_1^2 s_2^{-1} t s_2 t 
=
s_2  t  s_2 t s_1^2  t s_2 ts_2^{-1}
=
s_2  t  s_2 s_1^2t^2    s_2 ts_2^{-1}
 \in \AAA$.
 On the other hand,
 $ts_2  t  s_1^{-1}  s_2t   s_2 s_1 t s_2 t 
 =
 ts_2  s_1^{-1} t    s_2t   s_2 s_1 t s_2 t 
 =
 ts_2  s_1^{-1}     s_2t   s_2 ts_1  ts_2 t 
 =
 ts_2  s_1^{-1}     s_2t   s_2 t^2s_1   s_2 t
 \in R  ts_2  s_1^{-1}     s_2t   s_2 t s_1   s_2 t
 +
 R  ts_2  s_1^{-1}     s_2t   s_2 s_1   s_2 t$.
 Then, 
$ts_2  s_1^{-1}     s_2t   s_2 s_1   s_2 t=
ts_2  s_1^{-1}     s_2t   s_1 s_2   s_1 t
=ts_2  s_1^{-1}     s_2 t s_1   s_2   t s_1 \in \AAA$
by lemma \ref{lemtuuutuut}, and
$ts_2  s_1^{-1}     s_2t   s_2 t s_1   s_2 t = 
ts_2  s_1^{-1}    t s_2t   s_2  s_1   s_2 t
=
ts_2  s_1^{-1}    t s_2t   s_1  s_2   s_1 t
=
ts_2  s_1^{-1}    t s_2t   s_1  s_2    ts_1
=
ts_2      ts_1^{-1} s_2 s_1t    s_2    ts_1
=
ts_2      ts_2 s_1 s_2^{-1}t    s_2    ts_1
=
s_2      ts_2 ts_1 s_2^{-1}t    s_2    ts_1
=
s_2      ts_2 ts_1 t    s_2    ts_2^{-1}s_1
=
s_2      ts_2 s_1 t^2    s_2    ts_2^{-1}s_1
\in \AAA$.
\end{proof}

\subsection{Properties of the bimodule $\hat{B}$}

%\begin{lemme} \label{lemstAst} $\langle s_2,t \rangle A_3 \langle s_2, t \rangle \subset \hat{B}$.
\begin{proposition} \label{lemstAst}{\ } 
\begin{enumerate}
\item $\hat{B}$ is stable under $\phi$ and $\psi$.
\item $\hat{B} = \AAA +A_3 t s_2 t s_1 s_2^{-1} s_1 t s_2 t  + A_3 t^{-1} s_2^{-1} t^{-1} s_1^{-1} s_2 s_1^{-1} t^{-1} s_2^{-1} t^{-1}$
\item $\hat{B} = A_3 \langle s_2, t \rangle A_3 \langle s_2, t \rangle A_3$
\end{enumerate}

%$\langle s_2,t \rangle A_3 \langle s_2, t \rangle \subset \hat{B}$.
\end{proposition}
%\end{lemme}
\begin{proof}
Recall that
$
\hat{B} = \AAA + A_3 C^2 + A_3 C^{-2}
$. Since $C \in A_3^{\times} t s_2 s_1 t s_2 t$
and $C = (t s_2 s_1)^3$, we have
$C^{-1} = s_1^{-1} s_2^{-1} t^{-1} s_1^{-1} s_2^{-1} t^{-1} s_1^{-1} s_2^{-1} t^{-1}
= s_1^{-1} s_2^{-1}s_1^{-1} t^{-1}  s_2^{-1} s_1^{-1} t^{-1}  s_2^{-1} t^{-1}
= s_1^{-1} s_2^{-1}s_1^{-1} \phi(t  s_2 s_1 t  s_2 t)
\in A_3^{\times} \phi(A_3^{\times} C)$ hence
$C^{-1} \in A_3^{\times} \phi(C)$, and this implies
$C^{-2} \in A_3^{\times} \phi(C^2)$. From this we deduce
that $\hat{B}$ is $\phi$-stable. Moreover,
$\phi \circ \psi(C) \in t s_2 t s_1 s_2 t A_3^{\times} = t s_2  s_1 ts_2 t A_3^{\times}  \in C A_3^{\times}$
hence $\hat{B}$ is $\phi \circ \psi$-stable, and thus also $\psi$-stable,
that is (1).
An immediate consequence of lemma \ref{lemC2modA2} and of
$C^{-2} \in A_3^{\times} \phi(C^2)$ is then that
$\hat{B} =  \AAA +  A_3 t s_2 t s_1 s_2^{-1} s_1 t s_2 t A_3 + A_3 t^{-1} s_2^{-1} t^{-1} s_1^{-1} s_2 s_1^{-1} t^{-1} s_2^{-1} t^{-1} A_3 
= \AAA +  A_3 t s_2 t s_1 s_2^{-1} s_1 t s_2 t  + A_3 t^{-1} s_2^{-1} t^{-1} s_1^{-1} s_2 s_1^{-1} t^{-1} s_2^{-1} t^{-1}$.
Since $t^{-1} \in R^{\times} t +R$, from
lemma \ref{lemtuuutuut} one gets
$\hat{B} =  \AAA +  A_3 t s_2 t s_1 s_2^{-1} s_1 t s_2 t A_3 + A_3 t s_2^{-1} t s_1^{-1} s_2 s_1^{-1} t s_2^{-1} t A_3  
= \AAA +  A_3 t s_2 t s_1 s_2^{-1} s_1 t s_2 t  + A_3 t s_2^{-1} t s_1^{-1} s_2 s_1^{-1} t s_2^{-1} t$
and (2).
From this, and because of proposition \ref{lemstust}, (3) is equivalent
to $\langle s_2,t \rangle A_3 \langle s_2, t \rangle \subset \hat{B}$,
and we prove this now.

We have $\langle s_2,t \rangle \subset u_2 + u_2 t u_2 + \sum_{\eps \in \{ -1, 1 \}} u_2 t s_2^{\eps} t$
hence $\langle s_2,t \rangle A_3 \langle s_2,t \rangle \subset \hat{A}^{(2)} + u_2 t A_3 \langle s_2,t \rangle
+ \sum_{\eps \in \{-1, 1 \}} u_2 t s_2^{\eps} t A_3 \langle s_2,t \rangle$.
Since $\langle s_2,t \rangle \subset u_2 + u_2 t u_2 + \sum_{\eps \in \{ -1,1 \}} t s_2^{\eps} t u_2$,
$u_2 t A_3 \langle s_2,t \rangle \subset \hat{A}^{(2)} + \sum_{\eps \in \{-1,1 \}} u_2 t A_3 t s_2^{\eps}  t u_2$
and $u_2 t A_3 t s_2^{\eps} t u_2 \subset u_2 t u_1 u_2 u_1 u_2 t s_2^{\eps} t u_1 \subset A_3 t u_2 u_1 u_2 t s_2^{\eps} t u_1 \subset \AAA \subset \hat{B}$.
Moreover, $ts_2^{\eps} t A_3 \langle s_2,t \rangle \subset \hat{A}^{(2)} + t s_2^{\eps} t A_3 t u_2 + \sum_{\eta \in \{-1,1 \}}
t s_2^{\eps} t A_3 t s_2^{\eta} t u_2$. We have $t s_2^{\eps} t A_3 t  \in \AAA$ by proposition \ref{lemstust} ;
$t s_2^{\eps} t A_3 t s_2^{\eta} t \subset t s_2^{\eps} t^{\eps} A_3 t s_2^{\eta} t + \AAA \subset
t^{\eps} s_2^{\eps} t^{\eps} A_3 t s_2^{\eta} t + \AAA$ by proposition \ref{lemstust}, applying $t \in R t^{\eps} + R$ two times.
Similarly, one gets $t^{\eps} s_2^{\eps} t^{\eps} A_3 t s_2^{\eta} t \in
t^{\eps} s_2^{\eps} t^{\eps} A_3 t^{\eta} s_2^{\eta} t^{\eta} + \AAA$.
Whatever the choice of $\alpha \in \{ -1,1 \}$, one has $A_3 = u_2 s_1^{\alpha} s_2^{-\alpha} s_1^{\alpha} + u_2 u_1 u_2$
and $$
\begin{array}{lcl} t^{\eps} s_2^{\eps} t^{\eps} A_3 t^{\eta} s_2^{\eta} t^{\eta} & \subset&
t^{\eps} s_2^{\eps} t^{\eps} u_2 s_1^{\alpha} s_2^{-\alpha} s_1^{\alpha} t^{\eta} s_2^{\eta} t^{\eta}
+t^{\eps} s_2^{\eps} t^{\eps} u_2 u_1 u_2 t^{\eta} s_2^{\eta} t^{\eta} \\
& \subset&
u_2 t^{\eps} s_2^{\eps} t^{\eps}  s_1^{\alpha} s_2^{-\alpha} s_1^{\alpha} t^{\eta} s_2^{\eta} t^{\eta}
+u_2t^{\eps} s_2^{\eps} t^{\eps}  u_1  t^{\eta} s_2^{\eta} t^{\eta} u_2 \\
& \subset&
u_2 t^{\eps} s_2^{\eps} t^{\eps}  s_1^{\alpha} s_2^{-\alpha} s_1^{\alpha} t^{\eta} s_2^{\eta} t^{\eta}+
\AAA \\
%+u_2t^{\eps} s_2^{\eps} t^{\eps}  u_1  t^{\eta} s_2^{\eta} t^{\eta} u_2 \\
\end{array}$$
by proposition \ref{lemstust}, so we need to prove $t^{\eps} s_2^{\eps} t^{\eps}  s_1^{\alpha} s_2^{-\alpha} s_1^{\alpha} t^{\eta} s_2^{\eta} t^{\eta} \in \hat{B}$
for a suitable choice of $\alpha \in \{-1,1 \}$.

If $\eta = - \eps$ we take $\alpha = \eps$. Then 
$t^{\eps} s_2^{\eps} t^{\eps}  s_1^{\alpha} s_2^{-\alpha} s_1^{\alpha} t^{\eta} s_2^{\eta} t^{\eta}
=
t^{\eps} s_2^{\eps} t^{\eps}  s_1^{\eps} s_2^{-\eps} s_1^{\eps} t^{-\eps} s_2^{-\eps} t^{\eps}$
and, up to applying $\phi$, we can assume $\eps = 1$. Then 
$t^{\eps} s_2^{\eps} t^{\eps}  s_1^{\eps} s_2^{-\eps} s_1^{\eps} t^{-\eps} s_2^{-\eps} t^{\eps}
=t s_2 t  s_1 s_2^{-1} s_1 t^{-1} s_2^{-1} t
=t s_2   s_1 t s_2^{-1}t^{-1} s_1  s_2^{-1} t$.
From $t s_2 t s_2 = t s_2 t s_2$ one gets
$t s_2^{-1} t^{-1} = s_2^{-1} t^{-1} s_2^{-1} t s_2$,
hence 
$$
\begin{array}{lccccc}
& t s_2   s_1 t s_2^{-1}t^{-1} s_1  s_2^{-1} t &
=&t (s_2   s_1 s_2^{-1}) t^{-1} s_2^{-1} t s_2 s_1  s_2^{-1} t
&=&t s_1^{-1}   s_2 s_1 t^{-1} s_2^{-1} t s_2 s_1  s_2^{-1} t \\
=&s_1^{-1} t    s_2 s_1 t^{-1} s_2^{-1} t (s_2 s_1  s_2^{-1}) t
&=&s_1^{-1} t    s_2 s_1 t^{-1} s_2^{-1} t s_1^{-1} s_2  s_1 t
&=&s_1^{-1} t    s_2  t^{-1}(s_1 s_2^{-1} s_1^{-1})t  s_2   ts_1 \\
=&s_1^{-1} t    s_2  t^{-1}s_2^{-1} s_1^{-1} s_2t  s_2   ts_1 
&\in& A_3 \langle s_2,t \rangle u_1 \langle s_2, t \rangle A_3 &\subset& \AAA\\
\end{array}
$$
by proposition \ref{lemstust}.

Otherwise, we have either $\eps = \eta = 1$, in which case we take $\alpha = 1$ 
and get
$t^{\eps} s_2^{\eps} t^{\eps}  s_1^{\alpha} s_2^{-\alpha} s_1^{\alpha} t^{\eta} s_2^{\eta} t^{\eta}
=
t s_2  t  s_1  s_2^{-1} s_1 t  s_2  t \in \hat{B}$,
or we have $\eps = \eta = -1$, in which case we take $\alpha = -1$
and get $\phi(t s_2  t  s_1  s_2^{-1} s_1 t  s_2  t) \in \hat{B}$. This concludes the proof.
\end{proof}

\subsection{Computation of $C^3$ modulo $\hat{B}$}

%Our aim is to compute $C^3$ modulo $\hat{B}$.

We first need to prove a few preliminary lemmas.
 \begin{lemma} {\ } \label{lemC3AC}
\begin{enumerate}
\item $t s_2 s_1^{-1} s_2 t s_2 t^2 s_1^{-1} s_2 t \in R^{\times} t s_2 s_1^{-1} s_2 t s_2 s_1^{-1} s_2 t + \hat{B}$
\item $t s_2 s_1^{-1} s_2 t s_2 t s_1^{-1} s_2 t \in  \hat{B}$
\item $t s_2 t s_1^{-1} s_2 s_1^3 t s_2 s_1^{-1} t s_2 t \in R^{\times} t s_2 t s_1^{-1} s_2   t s_2 s_1^{-1} t s_2 t + \hat{B}$
\item $t s_2 t s_1^{-1} s_2 s_1^{-1} t s_2 s_1^{-1} t s_2 t \in -  c^{-1}b  t s_2 t s_1^{-1}s_2t s_2 s_1^{-1} t s_2 t+ \hat{B}$
\item $t s_2 t s_1^{-1} s_2 s_1 t s_2 s_1^{-1} t s_2 t \in  \AAA \subset \hat{B}$
%\item $t s_2 t s_1^{-1} s_2   t s_2 s_1^{-1} t s_2 t \in  \hat{B}$
\end{enumerate}
\end{lemma}
\begin{proof} 
Clearly (1) is a consequence of (2), since $t^2 \in R^{\times} + R t$, so we only need
to prove (2). Now
$t s_2 s_1^{-1} (s_2 t s_2 t) s_1^{-1} s_2 t 
= t s_2 s_1^{-1} t s_2 t (s_2  s_1^{-1} s_2) t$,
and $s_2 s_1^{-1} s_2 \in s_2^{-1} s_1 s_2^{-1} u_1 + u_1 u_2 u_1$,
hence
$t s_2 s_1^{-1} t s_2 t (s_2  s_1^{-1} s_2) t
\in t s_2 s_1^{-1} t s_2 t s_2^{-1}  s_1 s_2^{-1} tu_1 + t s_2 s_1^{-1} t s_2 t u_1 u_2 t u_1$.
We have $t s_2 s_1^{-1} (t s_2 t) s_2^{-1}  s_1 s_2^{-1} t 
= t (s_2 s_1^{-1} s_2^{-1})t s_2 t   s_1 s_2^{-1} t
= t s_1^{-1} s_2^{-1} s_1 t s_2 t   s_1 s_2^{-1} t
= s_1^{-1} t  s_2^{-1}  t (s_1 s_2  s_1)t   s_2^{-1} t
= s_1^{-1} t  s_2^{-1}  t s_2 s_1  s_2 t   s_2^{-1} t \in \AAA$ by proposition \ref{lemstust},
and $t s_2 s_1^{-1} t s_2 t u_1 u_2 t = t s_2  t s_1^{-1} s_2 u_1 t  u_2 t \subset  t s_2  t A_3 t  u_2 t \subset \hat{B}$
by proposition \ref{lemstAst}.

Since $s_1^3  = a s_1^2 + b s_1 + c$, we have
$s_1^2 = a s_1 + b + c s_1^{-1}$ hence
$s_1^3 = a(a s_1 + b + c s_1^{-1}) + b s_1 + c
= (a^2+b) s_1 + ac s_1^{-1} + (ab+c)$. Thus, since $c = (ab+c) + ac(-c^{-1} b) \in R^{\times}$,
(4) and (5) imply (3).

%\in c + R s_1^{-1} + R s_1$
%, (3) is an immediate consequence of (4), (5) and (6). 
We first
prove (5). We have $t s_2 t s_1^{-1} s_2 s_1 t s_2 s_1^{-1} t s_2 t 
= t s_2 t (s_1^{-1} s_2 s_1) t s_2 s_1^{-1} t s_2 t 
= (t s_2 t s_2) s_1 s_2^{-1} t s_2 s_1^{-1} t s_2 t 
= s_2 t s_2 t  s_1 s_2^{-1} (t s_2 t)s_1^{-1}  s_2 t 
= s_2 t s_2 t  s_1  (t s_2 t)s_2^{-1}s_1^{-1}  s_2 t 
= s_2 t s_2  t s_1  t s_2 ts_2^{-1}s_1^{-1}  s_2 t 
= s_2 t s_2   s_1  t^2 s_2 t(s_2^{-1}s_1^{-1}  s_2 )t 
= s_2 t s_2   s_1  t^2 s_2 ts_1s_2^{-1}  s_1^{-1} t 
= s_2 t s_2     t^2 s_1 s_2 s_1ts_2^{-1}  t s_1^{-1}
= s_2 t s_2     t^2 s_2 s_1 s_2ts_2^{-1}  t s_1^{-1}
 \in \AAA \subset \hat{B}$ by proposition \ref{lemstust}.

%(6) is clear, since 
%$t s_2 t s_1^{-1} s_2   t s_2 s_1^{-1} t s_2 t 
%=t s_2 t s_1^{-1} s_2   t s_2 s_1^{-1} t s_2 t  

We prove (4). 
From the study of $A_3$, we have that
$s_2^{-1} (s_1^{-1} s_2 s_1^{-1})
= (s_2^{-1} s_1^{-1} s_2) s_1^{-1}
= (s_1 s_2^{-1} s_1^{-1}) s_1^{-1}
= s_1 s_2^{-1} s_1^{-2}$.
Moreover,
$s_1^{-2} = c^{-1}s_1 - c^{-1}a -  c^{-1}b s_1^{-1}$
hence 
$s_2^{-1} (s_1^{-1} s_2 s_1^{-1}) = 
c^{-1}s_1 s_2^{-1}s_1 - c^{-1}as_1 s_2^{-1} -  c^{-1}b s_1 s_2^{-1}s_1^{-1}$.
It follows that
$t s_2 t s_1^{-1} s_2 s_1^{-1} t s_2 s_1^{-1} t s_2 t $ is equal to

\noindent $ s_2 s_2^{-1} (t s_2 t )s_1^{-1} s_2 s_1^{-1} t s_2 s_1^{-1} t s_2 t$ 
and thus to
%is equal to
$$
\begin{array}{cl}
 %=&
%s_2 s_2^{-1} (t s_2 t )s_1^{-1} s_2 s_1^{-1} t s_2 s_1^{-1} t s_2 t  \\
%=
 & s_2 t s_2 t (s_2^{-1}s_1^{-1} s_2 s_1^{-1}) t s_2 s_1^{-1} t s_2 t \\
= & s_2 t s_2 t( c^{-1}s_1 s_2^{-1}s_1 - c^{-1}as_1 s_2^{-1} -  c^{-1}b s_1 s_2^{-1}s_1^{-1})t s_2 s_1^{-1} t s_2 t\\
= &c^{-1} s_2 t s_2 t s_1 s_2^{-1}s_1t s_2 s_1^{-1} t s_2 t
-c^{-1}a s_2 t s_2 ts_1 s_2^{-1}t s_2 s_1^{-1} t s_2 t
-  c^{-1}b s_2 t s_2 ts_1 s_2^{-1}s_1^{-1}t s_2 s_1^{-1} t s_2 t\\
\end{array}
$$

We deal separately with each of these three terms. We prove
that the first one belongs to $\hat{B}$.
We have 
$$
\begin{array}{lclcl}
t s_2 t s_1 s_2^{-1}s_1t s_2 s_1^{-1} t s_2 t 
&=& t s_2 t s_1 s_2^{-1}t (s_1s_2 s_1^{-1}) t s_2 t 
&=& t s_2 t s_1 s_2^{-1}t s_2^{-1}s_1 (s_2 t s_2 t ) \\
&=& t s_2 t s_1 s_2^{-1}t s_2^{-1}s_1  t s_2 t s_2
&=& t s_2  s_1( t) s_2^{-1}t s_2^{-1}s_1  t s_2 t s_2 \\
&\in & R t s_2  s_1 t^{-1} s_2^{-1}t s_2^{-1}s_1  t s_2 t s_2
&+& R t s_2  s_1 s_2^{-1}t s_2^{-1}s_1  t s_2 t s_2
\end{array}
$$
Now $t (s_2  s_1 s_2^{-1})t s_2^{-1}s_1  t s_2 t
= t s_1^{-1}  s_2 s_1 t s_2^{-1}s_1  t s_2 t
=  s_1^{-1} t  s_2  ts_1 s_2^{-1}s_1  t s_2 t \in A_3 \langle s_2, t \rangle A_3 \langle s_2, t \rangle \subset \hat{B}$,
and similarly
$t s_2  s_1 t^{-1} s_2^{-1}(t) s_2^{-1}s_1  t s_2 t s_2
\in R t s_2  s_1 t^{-1} s_2^{-1}t^{-1} s_2^{-1}s_1  t s_2 t s_2  
+ R t s_2  s_1 t^{-1} s_2^{-1}s_2^{-1}s_1  t s_2 t s_2$
with $t s_2  s_1 t^{-1} s_2^{-1}s_2^{-1}s_1  t s_2 t  
= t s_2   t^{-1} s_1s_2^{-2}s_1  t s_2 t   \in \hat{B}$.
Finally,
$$
\begin{array}{lclcl}
 t s_2  s_1 (t^{-1} s_2^{-1}t^{-1} s_2^{-1})s_1  t s_2 t &
=& t s_2  s_1  s_2^{-1}t^{-1} s_2^{-1}t^{-1}s_1  t s_2 t
&=& t (s_2  s_1  s_2^{-1})t^{-1} s_2^{-1}s_1   s_2 t \\
&=&  t s_1^{-1}  s_2  s_1 t^{-1} s_2^{-1}s_1   s_2 t
&=&  s_1^{-1} t  s_2  s_1 t^{-1} s_2^{-1}s_1   s_2 t \in \AAA.
\end{array}
$$

We now turn to the second one. We have
$t s_2 ts_1 s_2^{-1}t s_2 s_1^{-1} t s_2 t 
= t s_2 ts_1 s_2^{-1}(t s_2 t)s_1^{-1}  s_2 t
= t s_2 ts_1 t s_2 ts_2^{-1}s_1^{-1}  s_2 t
= t s_2 ts_1 t s_2 t(s_2^{-1}s_1^{-1}  s_2) t
= t s_2 ts_1 t s_2 ts_1s_2^{-1}  s_1^{-1} t
= t s_2 ts_1 t s_2 ts_1s_2^{-1}   ts_1^{-1} 
\in \AAA$. We finally turn to the third one. We
have $t s_2 ts_1 s_2^{-1}s_1^{-1}t s_2 s_1^{-1} t s_2 t
= t s_2 t(s_1 s_2^{-1}s_1^{-1})t s_2 s_1^{-1} t s_2 t
= (t s_2 t)s_2^{-1} s_1^{-1}s_2t s_2 s_1^{-1} t s_2 t
= s_2^{-1} t s_2 t s_1^{-1}s_2t s_2 s_1^{-1} t s_2 t
= s_2^{-1} t s_2 t s_1^{-1}s_2t s_2 s_1^{-1} t s_2 t
$. Altogether this proves $t s_2 t s_1^{-1} s_2 s_1^{-1} t s_2 s_1^{-1} t s_2 t 
\in \hat{B} -  c^{-1}b s_2  s_2^{-1} t s_2 t s_1^{-1}s_2t s_2 s_1^{-1} t s_2 t
= \hat{B} -  c^{-1}b  t s_2 t s_1^{-1}s_2t s_2 s_1^{-1} t s_2 t$, hence (4).

\end{proof}

\begin{lemma} \ \label{lemC3D}
\begin{enumerate}
\item $t s_2 t s_1 s_2^{-1} s_1 t s_2^2 s_1 t s_2 t \in R^{\times} t s_2 t s_1 s_2^{-1} s_1 t s_2^{-1} s_1 t s_2 t + \hat{B}$
\item $t s_2 t s_1 s_2^{-1} s_1 t s_2 s_1 t s_2 t \in \AAA \subset \hat{B}$
\item $t s_2 t s_1 s_2^{-1} s_1t s_1  t s_2 t \in \hat{B}$
\end{enumerate}

\end{lemma}
\begin{proof} Since $s_2^2 \in R^{\times} s_2^{-1} +R s_2 + R$, (1) is a consequence of (2) and (3). We first prove (2).
We have 
$$
\begin{array}{clclcl}
& t s_2 t s_1 s_2^{-1} s_1 t s_2 s_1 t s_2 t
&=& t s_2 t s_1 s_2^{-1}  t (s_1s_2 s_1) t s_2 t
&=& t s_2 t s_1 s_2^{-1}  t s_2s_1 (s_2 t s_2 t) \\
= & t s_2 t s_1 s_2^{-1}  t s_2s_1 ts_2 t s_2 
&=& t s_2 t s_1 s_2^{-1}  t s_2s_1 (ts_2 t) s_2^{-1}s_2^2 
&=& t s_2 t s_1 s_2^{-1}  t (s_2s_1 s_2^{-1})ts_2 t s_2^2  \\
=&  t s_2 t s_1 s_2^{-1}  t s_1^{-1}s_2 s_1ts_2 t s_2^2 
&=& t s_2 t (s_1 s_2^{-1}  s_1^{-1})t s_2 s_1ts_2 t s_2^2 
&=& (t s_2 t) s_2^{-1} s_1^{-1}  s_2 t s_2 s_1ts_2 t s_2^2  \\
=& s_2^{-1} t s_2 t  s_1^{-1}  s_2 t s_2 s_1ts_2 t s_2^2 
&=& s_2^{-1} t s_2 t  s_1^{-1}  (s_2 t s_2 t )s_1s_2 t s_2^2 
&=& s_2^{-1} t s_2 t  s_1^{-1}  ts_2 t s_2  s_1s_2 t s_2^2  \\
= &s_2^{-1} t s_2   s_1^{-1}  t^2 s_2 t (s_2  s_1s_2) t s_2^2 
&=& s_2^{-1} t s_2   s_1^{-1}  t^2 s_2 t s_1  s_2s_1 t s_2^2 
&=& s_2^{-1} t s_2   s_1^{-1}  t^2 s_2 t s_1  s_2t s_1 s_2^2
\end{array}
 $$
and, since $t^2 \in R t + R$,
$t s_2   s_1^{-1}  t^2 s_2 t  s_1s_2 t \in
R t s_2   s_1^{-1}  t s_2 t    s_1s_2 t 
+ R t s_2   s_1^{-1}   s_2 t    s_1s_2 t $ ;
we have 
$t s_2   s_1^{-1}   s_2 t   s_1s_2 t 
\in \AAA$ by lemma \ref{lemtuuutuut}, 
and 
$t s_2   s_1^{-1}  t s_2 t   s_1s_2 t
= t s_2     t (s_1^{-1} s_2  s_1)t  s_2 t
= t s_2     t s_2 s_1  s_2^{-1} t  s_2 t \in \AAA$
by proposition \ref{lemstust}. This proves (2).
Then
(3) follows from $t s_2 t s_1 s_2^{-1} s_1t s_1  t s_2 t
= t s_2 t s_1 s_2^{-1} s_1^2t^2 s_2 t \in \langle s_2, t \rangle A_3 \langle s_2, t \rangle \subset \hat{B}$.
\end{proof}

\begin{lemma} \label{lemC3E} \ 
\begin{enumerate}
\item $t s_2 t s_1 s_2^{-1} s_1 t s_2 t^2 s_2 s_1 t s_2 t \in R^{\times} 
t s_2 t s_1 s_2^{-1} s_1 t s_2^2 s_1 t s_2 t +
\hat{B}$
\item $t s_2 t s_1 s_2^{-1} s_1 t s_2 t s_2 s_1 t s_2 t \in \hat{B}$.
\item $t s_2  s_1^2 t s_2 t s_2  s_1  t \in \AAA$
\item $t s_2  s_1^2 t s_2 t s_2  s_1 s_2^{-1} t s_2 t \in \hat{B}$.
%\item 
\end{enumerate}

\end{lemma}
\begin{proof}
Since $t^2 = d t + e \in R t + R^{\times}$, (1) is an immediate consequence of (2).
We prove (2). We have $t s_2 t s_1 s_2^{-1} s_1 t s_2 t s_2 s_1 t s_2 t
= t s_2 t s_1 s_2^{-1} s_1 (t s_2 t s_2) s_1 t s_2 t
= t s_2 t s_1 s_2^{-1} s_1  s_2 t s_2t  s_1 t s_2 t
= t s_2 t s_1 (s_2^{-1} s_1  s_2) t s_2 s_1t^2 s_2 t
= t s_2 t s_1 s_1 s_2  s_1^{-1} t s_2 s_1t^2 s_2 t
= t s_2  s_1^2 t s_2 t( s_1^{-1}  s_2 s_1)t^2 s_2 t
= t s_2  s_1^2 t s_2 t s_2  s_1 s_2^{-1} t^2 s_2 t$.
Using $t^2 \in R t + R$, we see that (2) is a consequence of (3) and (4).
%We prove (3). 
We have $t s_2  s_1^2 t s_2 t s_2  s_1  t  
= t s_2  s_1^2 t s_2 t s_2    ts_1 \in \AAA$ by proposition \ref{lemstust} and this
proves (3). We turn to (4). We have 
$t s_2  s_1^2 (t s_2 t s_2)  s_1 s_2^{-1} t s_2 t
= t s_2  s_1^2  s_2 t s_2 t   s_1 s_2^{-1} (t s_2 t)
= t s_2  s_1^2  s_2 t s_2 t   s_1  t s_2 ts_2^{-1}
= t s_2  s_1^2  s_2 t s_2    s_1  t^2 s_2 ts_2^{-1}$.
Using $t^2 \in R t + R$,
we only need to prove
$t s_2  s_1^2  s_2 t s_2    s_1  t s_2 t  \in \hat{B}$,
since $t s_2  s_1^2  s_2 t s_2    s_1   s_2 t =
t s_2  s_1^2  s_2 t (s_2    s_1   s_2) t
= t s_2  s_1^2  s_2 t s_1    s_2   s_1 t
= t s_2  s_1^2  s_2 s_1 t     s_2    ts_1
\in \langle s_2,t \rangle A_3 \langle s_2, t \rangle A_3 \subset \hat{B}$.
But $t s_2  s_1^2  s_2 t s_2    s_1  t s_2 t
= t s_2  s_1^2  (s_2 t s_2  t)  s_1   s_2 t
= t s_2  s_1^2  ts_2 t s_2    s_1   s_2 t
= t s_2t  s_1^2  s_2 t (s_2    s_1   s_2) t
= t s_2t  s_1^2  s_2 t s_1    s_2   s_1 t
= t s_2t  s_1^2  s_2 s_1t     s_2    ts_1
\in \langle s_2,t \rangle A_3 \langle s_2,t \rangle A_3 \subset \hat{B}$,
and this concludes the proof of (4), and thus  of the lemma.
\end{proof}

\begin{proposition} \label{propC3X}
$
C^3 \in A_3^{\times} t s_2 s_1^{-1} s_2 t s_2 s_1^{-1} s_2 t  + \hat{B}
$
\end{proposition}
\begin{proof}
Since $C^3$ is central and $\hat{B}$ is a $A_3$-bimodule, it is sufficient to prove
$C^3 \in A_3^{\times} t s_2 s_1^{-1} s_2 t s_2 s_1^{-1} s_2 t A_3^{\times} + \hat{B}$.
We have already proved
$
(s_1 s_2 s_1)^{-1} C^3 \in A_3^{\times} t s_2 t s_1 s_2^{-1} s_1 t s_2 t^2 s_2 s_1 t s_2 t (s_1 s_2 s_1)  + A_3 t s_2 s_1^{-1} t s_2 t s_1 C + A_3 C^2  + \hat{A}^{(2)} \subset A_3^{\times} t s_2 t s_1 s_2^{-1} s_1 t s_2 t^2 s_2 s_1 t s_2 t A_3^{\times} + \hat{B}
$
by lemma \ref{lemSimpl0C3}, so we only need to prove 
$t s_2 t s_1 s_2^{-1} s_1 t s_2 t^2 s_2 s_1 t s_2 t \in A_3^{\times} t s_2 s_1^{-1} s_2 t s_2 s_1^{-1} s_2 t A_3^{\times} + \hat{B}$.
Now 
$$
\begin{array}{lll}
& t s_2 t s_1 s_2^{-1} s_1 t s_2 t^2 s_2 s_1 t s_2 t \\
\in & 
t s_2 t s_1 s_2^{-1} s_1 t s_2^2 s_1 t s_2 t  + \hat{B} & \mbox{(by lemma \ref{lemC3E})}\\
\subset  & s_2^{-1} (s_2 t s_2 t) s_1 s_2^{-1} s_1 t s_2^{-1} s_1 t s_2 t  + \hat{B} & \mbox{(by lemma \ref{lemC3D})}\\
=  & s_2^{-1} ts_2 t (s_2  s_1 s_2^{-1}) s_1 t s_2^{-1} s_1 t s_2 t + \hat{B} \\
= &s_2^{-1} ts_2 t s_1^{-1}  s_2 s_1 s_1 t s_2^{-1} s_1 t s_2 t s_2 s_2^{-1} + \hat{B}\\
= & s_2^{-1} ts_2 t s_1^{-1}  s_2 s_1 s_1 t (s_2^{-1} s_1 s_2) t s_2 t  s_2^{-1} + \hat{B}\\
= & s_2^{-1} ts_2 t s_1^{-1}  s_2 s_1 s_1 t(s_1 s_2 s_1^{-1} t s_2 t  s_2^{-1} + \hat{B}\\
= & s_2^{-1} ts_2 t s_1^{-1}  s_2 s_1^3 t s_2 s_1^{-1} t s_2 t  s_2^{-1} + \hat{B}\\
\subset & 
 s_2^{-1} ts_2 t s_1^{-1}  (s_2  t s_2  t) s_1^{-1} s_2 t  s_2^{-1}  + \hat{B} & \mbox{(by lemma \ref{lemC3AC} (3))}\\
= &  s_2^{-1} ts_2 t s_1^{-1}  ts_2  t s_2   s_1^{-1} s_2 t  s_2^{-1} + \hat{B} \\
= & s_2^{-1} ts_2 s_1^{-1} t^2   s_2  t s_2   s_1^{-1} s_2 t  s_2^{-1} + \hat{B} \\
= & s_2^{-1} ts_2 s_1^{-1}    s_2  t s_2 t^2  s_1^{-1} s_2 t  s_2^{-1} + \hat{B} \\
\subset &
  s_2^{-1} ts_2 s_1^{-1}    s_2  t s_2   s_1^{-1} s_2 t  s_2^{-1}   + \hat{B} & \mbox{(by lemma \ref{lemC3AC} (1))}\\
\end{array}
$$
and this concludes the proof of the proposition.
\end{proof}

\subsection{Conclusion of the proof}

\begin{proposition} {\ }
\begin{enumerate}
\item $\hat{A}^{(3)} t \subset \AAA t + \hat{A}^{(3)}$
\item $\AAA t \subset \hat{A}^{(3)} + \hat{B} $
\item $\hat{B} t \subset \hat{A}^{(3)} + \hat{B}$
\item $\hat{A}^{(4)} = \hat{A}^{(3)} + \hat{B}$
\item $\hat{A}^{(4)} = \hat{A}^{(5)} = \hat{A} = \hat{A}^{(3)} + A_3 C^2 + A_3 C^{-2}$
\item $\hat{A}^{(4)} = \AAA + A_3 C^2 + A_3 C^{-2} + A_3 C^3$
\end{enumerate}
\end{proposition}
\begin{proof}
In order to simplify notations, we let 
$X = t s_2 s_1^{-1} s_2 t s_2 s_1^{-1} s_2 t$, 
$Y_+ = t s_2 t s_1  s_2^{-1} s_1t s_2 t $,
$Y_- = t s_2^{-1} t s_1^{-1} s_2 s_1^{-1} t s_2^{-1} t$.
By proposition  \ref{propC3X}, we have $C^3 \in A_3^{\times} X + \hat{B}$,
hence $X \in A_3^{\times} C^3 + \hat{B} = A_3^{\times} C^3 + A_3 C^2 + A_3 C^{-2} + \AAA$.
From this we deduce that, for all $m \in A_3$, $m X - X m \in \AAA$,
hence $A_3 X A_3 +\AAA= A_3 X + \AAA$,
and $\hat{A}^{(3)} = \AAA + A_3 X A_3 = \AAA + A_3 X$.
It follows that $\hat{A}^{(3)} t \subset \AAA t + A_3 X t$,
and clearly  $X t \in \hat{A}^{(3)}$, whence
$\hat{A}^{(3)} t \subset \AAA t + \hat{A}^{(3)}$ and (1).
%\end{proof}
On the other hand,
$$
\hat{A}^{(2 \frac{1}{2})} = \hat{A}^{(2)} + A_3 t s_2 s_1 t s_2 t A_3
+ A_3 t s_2 s_1 t s_2^{-1} t A_3%$$ {} $$
+ A_3 t s_2 s_1^{-1} t s_2^{-1} t A_3
+ A_3 t s_2^{-1} s_1^{-1} t s_2^{-1} t A_3
$$
hence
%donc
$$
\begin{array}{lcl}
\hat{A}^{(2 \frac{1}{2})} t &\subset& \hat{A}^{(2)}t + A_3 t s_2 s_1 t s_2 t A_3t
+ A_3 t s_2 s_1 t s_2^{-1} t A_3t 
+ A_3 t s_2 s_1^{-1} t s_2^{-1} t A_3t
+ A_3 t s_2^{-1} s_1^{-1} t s_2^{-1} t A_3t \\
\hat{A}^{(2 \frac{1}{2})} t &\subset& \hat{A}^{(3)} + A_3 t s_2 s_1 t s_2 t A_3t
+ A_3 t s_2 s_1 t s_2^{-1} t A_3t
+ A_3 t s_2 s_1^{-1} t s_2^{-1} t A_3t
+ A_3 t s_2^{-1} s_1^{-1} t s_2^{-1} t A_3t
\end{array}
$$

We have $A_3 t s_2 s_1 t s_2 t A_3t =
A_3 C A_3t  = A_3 C t = A_3 t s_2 s_1 t s_2 t^2 \subset \hat{A}^{(3)}$,
and similarly $ t s_2^{-1} s_1^{-1} t s_2^{-1} t \in A_3^{\times} C^{-1}+
\hat{A}^{(2)}$ implies $A_3t s_2^{-1} s_1^{-1} t s_2^{-1} tA_3 t
\subset A_3 t s_2^{-1} s_1^{-1} t s_2^{-1} t^2 + \hat{A}^{(2)} t
\subset \hat{A}^{(3)}$.

Moreover, by lemma \ref{lemA3cmod},
$$
t s_2 s_1 t s_2^{-1} t A_3t \subset 
\hat{A}^{(2)} t
+ \sum_{a \in \{ -1,0,1 \}} \sum_{b,\eps \in \{ -1,1 \}} A_2 t s_2 s_1 t s_2^{\eps} t s_1^b s_2^a t
%$${} \\ $$
 +\sum_{\eps \in \{ -1,1 \}} \left( A_2 t s_2 s_1 t s_2^{\eps} t s_1 s_2^{-1} s_1 t+ A_2 t s_2 s_1  t s_2^{\eps} t^2 \right)$$
hence

$$
\begin{array}{lcl}
t s_2 s_1 t s_2^{-1} t A_3t &\subset& 
\hat{A}^{(3)} 
+ \sum_{a \in \{ -1,0,1 \}} \sum_{b,\eps \in \{ -1,1 \}} A_2 t s_2 s_1 t s_2^{\eps} t s_1^b s_2^a t
%$${} \\ $$
 +\sum_{\eps \in \{ -1,1 \}} \left( A_2 t s_2 s_1 t s_2^{\eps} t s_1 s_2^{-1} s_1 t \right) \\
& \subset & 
%t s_2 s_1 t s_2^{-1} t A_3t \subset 
\hat{A}^{(3)} 
+ \sum_{a \in \{ -1,0,1 \}} \sum_{b,\eps \in \{ -1,1 \}} A_2 t s_2  t s_1s_2^{\eps} s_1^bt  s_2^a t
%$${} \\ $$
 +\sum_{\eps \in \{ -1,1 \}} \left( A_2 t s_2  ts_1 s_2^{\eps} s_1 t  s_2^{-1}  ts_1 \right)
\end{array}
$$
and, by proposition \ref{lemstAst}, $t s_2 s_1 t s_2^{-1} t A_3t \subset \hat{A}^{(3)} + \hat{B}
=\hat{A}^{(3)}+ \AAA + A_3 C^2 + A_3 C^{-2}
=\hat{A}^{(3)}+  A_3 C^2 + A_3 C^{-2}$. 
This implies $\phi( t s_2 s_1^{-1} t s_2^{-1} t A_3t) \subset \phi(t s_2 s_1^{-1} t s_2^{-1} t) A_3 t^{-1}
\subset \hat{A}^{(3)}+  A_3 C^2 + A_3 C^{-2}$ by lemma \ref{lemA3cmodPhi} and the above, whence (4).

We have $\hat{B} \subset
\AAA +A_3 t s_2 t s_1 s_2^{-1} s_1 t s_2 t  + A_3 t^{-1} s_2^{-1} t^{-1} s_1^{-1} s_2 s_1^{-1} t^{-1} s_2^{-1} t^{-1}$
by proposition \ref{lemstAst}, and we already know
$\AAA t \subset \hat{A}^{(3)} + \hat{B}$. Since
$t^{-1} s_2^{-1} t^{-1} s_1^{-1} s_2 s_1^{-1} t^{-1} s_2^{-1} t^{-1}.t
\in \hat{A}^{(3)}$ and
$t s_2 t s_1 s_2^{-1} s_1 t s_2 t^2 \in
R t s_2 t s_1 s_2^{-1} s_1 t s_2 t + \hat{A}^{(3)}
\subset \hat{A}^{(3)} + \hat{B}$ because of $t^2 \in R t + R$,
this proves (3). From (3) and (4) we get
$\hat{A}^{(4)} t \subset \hat{A}^{(3)} t + \hat{B} t 
\subset \hat{A}^{(4)} + \hat{B} t \subset \hat{A}^{(4)} + 
\hat{A}^{(3)} + \hat{B} = \hat{A}^{(4)}$. This
implies $\hat{A}^{(5)} \subset \hat{A}^{(4)}$ hence $\hat{A}^{(4)} = \hat{A}^{(5)}$
and (5). 
As already noticed we have $\hat{A}^{(3)} = \AAA + A_3 X A_3$, and
$X \in A_3^{\times} C^3 + \hat{B}$ with $C^3$ central implies
$\hat{A}^{(4)} = \hat{A}^{(3)} + \hat{B} = \AAA + \hat{B}
+ A_3 C^3 = \AAA + A_3 C^2 + A_3 C^{-2} + A_3 C^3$ and (6).
\end{proof} 
%$t s_2 s_1^{-1} t s_2^{-1} t A_3t \subset \hat{A}^{(3)}+  A_3 C^2 + A_3 C^{-2}$.

\begin{corollary} As a $A_3$-module, $\hat{A}$ is generated by 54 elements.
\end{corollary}
\begin{proof}
By the above property, $\hat{A}$ is generated by $\AAA$ and 3 elements. By
proposition \ref{propX}, $\AAA$ is generated by $\hat{A}^{(2)}$ and $2 \times 9 = 18$
elements.
By proposition \ref{lemG26L1},
%(2011, lemme 2), 
$
\hat{A}^{(2)} = \hat{A}^{(1)} + A_3 t s_2 t A_3 + A_3 t s_2^{-1} t A_3 + A_3 t s_2 s_1^{-1} s_2 t A_3.
$
Since $ts_2t $ commutes with $u_2$, and because $A_3$ is a (free) $u_2$-module
of rank (at most) $8$,
%un $u_2$-module libre de rang $8$, 
we know that $A_3 t s_2 t A_3$ is generated as
a $A_3$-module by $8$ éléments.
Since $t \in R^{\times} t^{-1} + R$, we have
$A_3 t s_2^{-1} t A_3 + \hat{A}^{(1)} = 
A_3 t^{-1} s_2^{-1} t^{-1} A_3 + \hat{A}^{(1)}$ and, because
$t^{-1} s_2^{-1} t^{-1}$ commutes with $u_2$, this $A_3$-module
is generated by  $\hat{A}^{(1)}$ together with $8$ éléments. 
Finally, because $A_3 = u_1 s_2 s_1^{-1} s_2 + u_1 u_2 u_1$, we have
 $(t s_2 s_1^{-1} s_2 t) s_1 = 
t s_2 (s_1^{-1} s_2  s_1) t= 
t s_2^2 s_1  s_2^{-1} t \in t u_1 s_2 s_1^{-1} s_2 t + t u_1 u_2 u_1 t
\subset u_1 t  s_2 s_1^{-1} s_2 t +  u_1 t u_2  tu_1
\subset u_1 t  s_2 s_1^{-1} s_2 t +  A_3 t s_2  tA_3+ A_3 t s_2^{-1}  tA_3 + \hat{A}^{(1)}$.
Since $A_3$ is also a (free) $u_1$-module of rank $8$,
we deduced from this that $\hat{A}^{(2)}$ is generated by
$\hat{A}^{(1)} + A_3 t s_2  tA_3+ A_3 t s_2^{-1}  tA_3$ together with $8$ elements,
and it follows that $\hat{A}^{(2)}$ is generated by
%est engendré par 
$\hat{A}^{(1)}$ together with $2 \times 8 + 8 = 24$ elements.
Finally, $\hat{A}^{(1)} = A_3+A_3 t A_3$ is  $A_3$-generated par $1+8=9$ elements,
since $t$ commutes with $u_1$ and $A_3$ is generated by $8$ elements as a
$u_1$-module.  This proves that
%On en déduit que 
$\hat{A}$ is generated
as a $A_3$-module by $21 + 24 + 9 = 54$ elements. 
\end{proof}

Since $A_3$ is a $R$-module of rank $24$, this has for immediate
consequence the following, which proves theorem \ref{theoG26}.
\begin{corollary} As a $R$-module, $\hat{A}$ is spanned by $1296$ elements.
\end{corollary}

\begin{remark} The $R$-basis provided by this corollary is actually made
out of elements of $B$, and contains $1$.
\end{remark}

{\bf Acknowledgements.} I thank Marc Cabanes, Christophe Cornut, Jean Michel, Rapha\"el Rouquier and Olivier Schiffmann for useful discussions. I also thank Gunter Malle and Jean Michel for a careful reading of the first half of the paper.

%\tableofcontents

\end{document}